\documentclass[journal,12pt]{elsarticle}

\usepackage{lineno,hyperref}
\modulolinenumbers[5]

\usepackage{amsmath,amssymb,amsfonts}
\usepackage{amsthm}
\usepackage[ruled,linesnumbered]{algorithm2e}
\usepackage{caption}
\usepackage{subcaption}
\usepackage{mathrsfs}
\usepackage{booktabs}
\usepackage{bbding}
\usepackage{geometry}
\usepackage{color}
\usepackage{easyReview}

\journal{Journal of \LaTeX\ Templates}

\newtheorem{remark}{Remark}

\newtheorem{theorem}{Theorem}

\newtheorem{lemma}{Lemma}
\newtheorem{definition}{Definition}

\newtheorem{requirement}{Requirement}

\usepackage{multirow}

\def\d{\mathrm{d}}

\begin{document}
	\begin{frontmatter}
		\title{Dynamic Controlled Variables Based Dynamic Self-Optimizing Control}
		\author[mymainaddress,quzhou]{Chenchen Zhou}
		\author[mymainaddress,quzhou]{Shaoqi Wang}
		\author[mymainaddress]{Hongxin Su}
		\author[mymainaddress]{Xinhui Tang}
		\author[mymainaddress,quzhou]{Yi Cao}
		
		\author[mymainaddress,quzhou]{Shuang-Hua Yang\corref{mycorrespondingauthor}}
		\cortext[mycorrespondingauthor]{Corresponding author}
		\ead{yangsh@zju.edu.cn}
		
		\address[mymainaddress]{College of Chemical and Biological Engineering, Zhejiang University, Hangzhou, Zhejiang Province, P.R. China}
		\address[quzhou]{Institute of Zhejiang University-Quzhou, Quzhou, Zhejiang Province,  P.R. China}
		\date{}
		
		\begin{abstract}
			{Self-optimizing control is a strategy for selecting controlled variables, where the economic objective guides the selection and design of controlled variables, with the expectation that maintaining the controlled variables at constant values can achieve optimization effects, translating the process optimization problem into a process control problem. Currently, self-optimizing control is widely applied to steady-state optimization problems. However, the development of process systems exhibits a trend towards refinement, highlighting the importance of optimizing dynamic processes such as batch processes and grade transitions. This paper formally introduces the self-optimizing control problem for dynamic optimization, termed the dynamic self-optimizing control problem, extending the original definition of self-optimizing control. A novel concept, "dynamic controlled variables" (DCVs), is proposed, and an implicit control policy is presented based on this concept. The paper theoretically analyzes the advantages and generality of DCVs compared to explicit control strategies and elucidates the relationship between DCVs and traditional controllers. Moreover, this paper puts forth a data-driven approach to designing self-optimizing DCVs, which considers DCV design as a mapping identification problem and employs deep neural networks to parameterize the variables. Three case studies validate the efficacy and superiority of DCVs in approximating multi-valued and discontinuous functions, as well as their application to dynamic optimization problems with non-fixed horizons, which traditional self-optimizing control methods are unable to address.}
		\end{abstract}
		\begin{keyword}
			Self-Optimizing Control, dynamic optimization, implicit learning
		\end{keyword}
	\end{frontmatter}
	\section{Introduction}
	The field of control systems employs specialized techniques to regulate dynamic systems, achieving desired outcomes despite uncertainties\cite{alleyne2023}. Thus, in modern system, control system design is especially crucial, along with the design of the system and equipment. Broadly speaking, control system design can be divided into two parts: control structure design and controller design. 
	During the control structure design phase, engineers make critical decisions concerning the selection of variables for measurement, the determination of inputs to be manipulated, and the establishment of links between these two sets \cite{foss1973,skogestad2004,skogestad2005}.
	The design of the controller is usually carried out under the assumption of a predetermined control structure.
	
	The design of the control structure, as a step that embodies the inherent characteristics of the actual system, typically exerts a significant influence on controller design. A well-designed control structure can simplify the controller design and enhance control performance. Therefore, the interdependence and indispensability of both control structure design and controller design are widely acknowledged, with control structure design being deemed more critical\cite{foss1973,stephanopoulos2000}. 
	However, greater emphasis is still placed on controller design within the field of control  academic society. 
	
	
	In the domain of control structure design, Skogestad\cite{skogestad2000} is the pioneer in formulating the concept of a self-optimizing control (SOC) structure. This approach is distinguished by its emphasis on minimizing the loss of optimality relative to a defined economic objective function. Rather than focusing solely on controllability or control performance, see \textit{i.e.} \cite{vandewal2001}, SOC seeks to retain near-optimal operation even in the face of uncertainty or disturbances by selecting self-optimizing controlled variables (CVs):
	
	\textit{Self-optimizing control is when we can achieve an acceptable loss with constant setpoint values for the controlled variables (without the need to re-optimize when disturbances occur).}
	
	
	The successful application of self-optimizing control requires tools and methods for selecting good CVs. In recent years, substantial progress has been made in rigorously formulating mathematical optimization methods to identify self-optimizing controlled variables. 
	Initial local SOC approaches relied on linear approximations around nominal operating points, providing useful candidate control structures for limited regions \cite{halvorsen2003,alstad2007,kariwala2008}. Tailored branch-and-bound techniques expanded the search over all variable combinations to identify the best measurement subsets \cite{caoBidirectionalBranchBound2008,kariwalaBidirectionalBranchBound2009,kariwalaBidirectionalBranchBound2010}. Data-driven \cite{ye2013a}, model-free \cite{jaschke2013} variants extended these local methods by analyzing measurement data directly.
	
	Significant efforts to obtain near-optimal CVs over wider operating ranges have also emerged. These accommodate active set changes and smooth nonlinearities, establishing connections to related techniques like necessary conditions of optimality tracking (NCO tracking) for non-localized optimization \cite{jaschke2011}. In particular, polynomial\cite{jaschke2012} and data-driven methods \cite{ye2013a,Ye2015,ye2022a,ye2023} show promise for larger envelopes. Alternatively, switching control structures between active set regions can cover expanded operation \cite{manum2012,reyes-lua2020a,krishnamoorthy2020a}.
	
	Recent literature also demonstrates the wide applicability of self-optimizing control across diverse engineering domains. For example, Abunde Neba \textit{et al}. applied SOC to wastewater treatment process, optimizing anaerobic digestion systems under kinetic uncertainty \cite{abundeneba2020}. Bottari and Braccia optimized the economic performance of continuous cross-flow grain dryers by selecting the SOC control structure \cite{bottari2023}. Fu \textit{et al}. developed SOC methods to maximize the efficiency of solid oxide fuel cells, with both local and global approaches \cite{fu2023a,fu2023b}.
	Li \textit{et al}. proposed a two-stage anaerobic digestion control scheme combining self-optimizing and extremum-seeking techniques \cite{li2023c}. Extractive distillation, chillers, and other processes have also benefited from self-optimizing principles \cite{zhang2023a,zhao2021}.
	
	In summary, recent literature provides numerous examples of successful SOC applications across a diverse set of chemical, energy, industrial, and environmental systems.
	Both theoretical developments and expanded applications continue advancing SOC as a powerful approach for handling uncertainties and optimizing complex, dynamic systems. SOC has demonstrated wide versatility and effectiveness for automated process optimization. 
	
	Though SOC has achieved impressive successes, a glaring limitation persists: it is currently confined to disturbances or variations in plant behavior that persist over a significantly longer period than the plant dynamics. In other words, SOC primarily focuses on static optimization problems. 
	However, most real industrial processes exhibit dynamic behavior where disturbances can evolve within the timescale of the process dynamics.
	
	{The backdrop for the proposition of static self-optimizing control methodologies was large-scale continuous chemical processes, where such processes operate predominantly in steady-state conditions for most of the time. Consequently, the optimization problem for these processes can be regarded as a static optimization issue. However, the current process industry exhibits a trend towards refinement and eco-friendliness \cite{nikacevicOpportunitiesChallengesProcess2012}. The inclination towards refinement imposes more stringent demands on product quality and stability of chemicals, inevitably leading to reduced intermediate storage, increased recycle streams, and expanded downstream processing, thereby augmenting the complexity of plant control and exacerbating the challenges in product quality control \cite{nikacevicOpportunitiesChallengesProcess2012}. Simultaneously, the demand for refined products is relatively lower, and catering to diverse applications may necessitate alterations in the system's operational state, resulting in a transition from predominantly continuous processes to batch processes. Even for continuous processes, product grade transitions have become more frequent \cite{shiOptimizationGradeTransitions2016,kadamOptimalGradeTransition2007}. Both batch processes and grade transition processes deviate from steady-state conditions, rendering the system's dynamic behavior a critical and indispensable factor influencing the economic viability of the process system.}
	
	{Despite the practical need, extending SOC to handle dynamic disturbances and systems remains an open challenge \cite{jaschkeSelfoptimizingControlSurvey2017}.
	Currently, the prevalent approach to tackle dynamic optimization problems in process systems involves integrating optimization and control \cite{biegler2009,tatjewski2008}, adopting nonlinear model predictive control \cite{grune2017} or economic model predictive control\cite{ellis2017} strategies for dynamic optimization. These methods confront a dilemma – the more intricate the considered model, the better the online control performance, but the online solution becomes increasingly challenging. Here, model intricacy refers to incorporating more realistic assumptions, minimizing model simplifications, such as employing distributed parameter models\cite{dufour2003} or considering comprehensive uncertainties during the solution process, leading to online stochastic\cite{mesbah2016} or robust optimization\cite{lucia2014}. Self-optimizing control methods can alleviate this contradiction by shifting the complex optimization process offline to the design of controlled variables, eliminating the need for repeated online optimization and thereby reducing online computational demands. }
	
	{Regardless of whether it is the practical demands of the industrial sector or the theoretical advancement of self-optimizing control itself, there exists an pressing necessity to extend the concepts of self-optimizing control to encompass dynamic optimization problems.}
	While self-optimizing control has predominantly focused on static systems, some initial efforts have attempted to extend the methodology to dynamic processes: 
	Baldea \textit{et al}. provided early analysis of dynamic considerations for SOC synthesis \cite{baldeaDynamicConsiderationsSynthesis2008}. Grema and Cao developed a dynamic SOC approach for uncertain reservoir waterflooding \cite{grema2020} by extending the data-driven SOC in static problems\cite{ye2013a}. Hu \textit{et al.} proposed a framework to design controller and CVs together by solving a bio-objective optimization problem in linear systems \cite{hu2012}. Ye and Skogestad formulated a dynamic SOC framework for unconstrained batch processes \cite{ye2018}. Klemets \textit{et al.} designed PI controllers and CVs together \cite{Klemets2019}.
	
	{However, these works have focused on developing SOC techniques tailored to specific dynamic process models. In particular, for dynamic optimization problems with infinite and free time horizon, existing SOC methods are still incapable of providing solutions. A universally applicable framework for dynamic self-optimizing control remains an open challenge. 
	Research resolving core issues around integrating dynamics into control structure selection is critical to unlock the full potential of self-optimizing control \cite{engell2007a,ellis2014a}.  
	Developing rigorous dynamic SOC techniques promises to expand the scope and impact of this powerful methodology.}
	
	{Motivated by those reasons, in this paper, the dynamic SOC problem is systematically proposed. Since dynamic SOC concerns the dynamic behavior of the optimization system, requiring controlled variables to remain constant at all times, traditional controlled variables only describe the operational objectives the system aims to achieve, typically representing steady-state behavior. This prompts us to introduce a novel control strategy, termed dynamic controlled variables (DCVs).} DCVs extend the CV notion to dynamic systems by implicitly solving a parameterized system of functions over the control trajectory:
	\begin{equation}
		\hat{{u}_t}=\underset{{u}_t \in \mathcal{U}_t}{\operatorname{arg}} \quad h_\theta({u}_t,\boldsymbol{O}_t)=0
	\end{equation}
	where $\boldsymbol{O}_t$ and ${u}_t$ denote the observation and control input at time $t$.
	
	{To elucidate the characteristics of DCVs more clearly, this paper draws a comparison between the concept of DCVs, controllers, and sliding mode control. Furthermore, based on the definition of the dynamic SOC problem, three requirements for self-optimizing DCVs are proposed, and four indices are introduced to quantify these requirements. Finally, the design of DCVs is viewed as a mapping identification problem, and a method for designing self-optimizing DCVs is presented, with its superiority validated through simulation experiments.} 
	
	The contribution of this paper is summarized as follows:
	{The contribution of this paper is summarized as follows:
	\begin{enumerate}
		\item The dynamic self-optimizing control problem is formally introduced, enabling the controlled variable design problem to be considered and investigated from a novel perspective.
		\item The concept of DCVs is proposed, an implicit control policy based on DCVs is presented, and its superiority over explicit control policys is theoretically proven.
		\item  The relationships between DCVs and traditional controllers, as well as the distinctions between dynamic SOC and sliding mode control, are elucidated.
		\item A method to design self-optimizing DCVs is proposed, overcoming the limitations of current self-optimizing control methods which are incapable of solving dynamic optimization problems with infinite and free time horizons.
	\end{enumerate}
	More broadly, the proposed implicit control policy based on DCVs in this paper can be applied to topics revolving around the design of explicit control strategies, such as model predictive control policy approximations \cite{zhou2024a}, learning for control, and related areas.}

	The paper is organized as follows:
	Section~\ref{sec:DOI} reviews dynamic optimization problems and their implementation.
	Section~\ref{sec:DCV} proposes the dynamic self-optimizing control problem and delineates the formulation and definition of DCVs, and propounds an implicit control strategy based on DCVs, with theoretical analysis demonstrating the enhanced flexibility of the DCV representation for control policies. The relationship between DCVs and controllers, as well as the distinctions between dynamic SOC and sliding mode control, are also discussed.
	Section~\ref{sec:DSOC} presents a method for designing self-optimizing DCVs, utilizing a unified metric to evaluate the performance of DCVs.
	Section~\ref{sec:CS} investigates 4 case studies, validating the advantages of DCVs in approximating multivalued and discontinuous functions, in addition to applications of self-optimizing DCVs for dynamic optimization problems with non-fixed horizons.

	\section*{Notion}
	Define $\mathbb{Z}{[a,b]}$ as the set of integers between the interval $[a,b]$ and $\mathbb{R}^{n}$ as the $n$-dimensional real number space. Let $I_m$ denote the $m \times m$ identity matrix. Use $0_{m\times n}$ to represent a matrix of zeros, with subscript dimensions omitted when clear from context for simplicity. For a sequence $\{s_k\}_{k=0}^{L-1}$ with $s_k \in \mathbb{R}^\eta$, stack its vectors as $\boldsymbol{s} = [s_{0}^\top, s_{1}^\top, \dots, s_{L-1}^\top]^\top$ and a stacked window of it as $\boldsymbol{s}_{[l,j]} =[s_{l}^\top, s_{l+1}^\top, \dots, s_j^\top]^\top$ for $0 \leq l < j$. 
	
	All variables in this paper and their descriptions are in Table \ref{tab:var} in appendix.
	\section{Dynamic optimization and Implementation}
	\label{sec:DOI}
	The primary objective of this paper is to systematically extend the concept of self-optimizing control to dynamic optimization problem. Henceforth, dynamic optimization problems and their implementation strategies are introduced in this section.
	\subsection{Dynamic optimization}
	Before introducing to the dynamic optimization, let us first describe the system under consideration in this paper:
	\begin{equation}
		\label{eq:sys}
		{x}_{k+1} = f_k(x_{k},u_{k},d_{k})
	\end{equation}
	\begin{equation}
		\label{eq:sys_m}
		{y}_{k} = m_k(x_{k},u_{k},d_{k})+w_{k} ,
	\end{equation}
	where $k = 0,1,\dots,L$ denotes the time instant, $L\in [0,\inf]$ denotes the operating time horizon, $x_{k} \in \mathbb{R}^{n_x}$ is the system state at instant $k$, and the initial states $x_0 \in \mathcal{X}_0$, $u_{k}  \in \mathbb{R}^{n_u}$ is the system control input at instant $k$,	$f_k: \mathbb{R}^{n_x} \times \mathbb{R}^{n_u} \times \mathbb{R}^{n_d}\rightarrow \mathbb{R}^{n_x}$ are smooth functions. $d_k \in \mathbb{R}^{n_d} $ is random disturbance which follows the distribution $d_k \sim \mathcal{D}_k$.  $y_{k} \in \mathbb{R}^{n_y}$ are the measurements, $m_k: \mathbb{R}^{n_x} \times \mathbb{R}^{n_u} \times \mathbb{R}^{n_d} \rightarrow \mathbb{R}^{n_y}$ is the measurement function, and $w_{k} \in \mathbb{R}^{n_y}$ denotes the measurement noise, which follows the distribution $w_k \sim \mathcal{W}$. 
	
	The general formulation of dynamic optimization problems shall be subsequently delineated.
	When optimizing operations for a given system, the control objectives should be firstly formulated. One category relates to feasibility - keeping process variables within acceptable limits despite disturbances. These constraints may arise from product quality, safety, environmental regulations, etc. In this paper, the following constrains are considered:
	\begin{equation}
		{g}_k(u_k, x_k) \leq 0 \ , k = 0,1,2,\dots,L  \label{eq:cons_path}
	\end{equation}
	where $g_k: \mathbb{R}^{n_u} \times \mathbb{R}^{n_x} \rightarrow \mathbb{R}^{n_{g_k}}$ is a known smooth function representing the constraint on the control input $u_k$ and state $x_k$ at time $k$. Or more compactly:
	\begin{equation}
		\boldsymbol{g}(\boldsymbol{u}, x_0, \boldsymbol{d}) \leq 0
	\end{equation}
	where $\boldsymbol{g}: \mathbb{R}^{n_u\times L} \times \mathbb{R}^{n_x} \times \mathbb{R}^{n_d\times L} \rightarrow \mathbb{R}^{n_{g}}$ is a known smooth function representing the constraints over the control sequence $\boldsymbol{u}$ given initial state $x_0$ and disturbance sequence $\boldsymbol{d}$.  $n_{g} = \sum_{k=0}^{L} n_{g_k}$
	$\boldsymbol{d}$ and $\boldsymbol{u}$ denote the disturbance sequence and control input sequence over some horizon of length $L \in [0,\inf]$, respectively:
	$$
	\boldsymbol{d} := \left[d_0^{\top}  \ d_{1}^{\top}  \ \dots \  d_{L-1}^{\top}  \right]^{\top} \in  \mathbb{R}^{n_d  L} ,
	$$
	$$
	\boldsymbol{u} := \left[u_0^{\top}  \ u_{1}^{\top}  \ \dots \  u_{L-1}^{\top}  \right]^{\top} \in  \mathbb{R}^{n_u  L} .
	$$
	%
	
	The second objective category stems from economics - optimizing performance post-feasibility. In this paper, it is defined through a cost function of the form
	\begin{equation}
		\label{eq:obj}
		\mathcal{J}(\boldsymbol{u} | x_0,\boldsymbol{d}) :=  l_f(x_{L})+ \sum_{i=0}^{L-1} l_i(x_{i},u_{i})
	\end{equation}
	where $l_i: \mathbb{R}^{n_x} \times \mathbb{R}^{n_u} \times \mathbb{R}^{n_d} \rightarrow \mathbb{R}$ and $l_f: \mathbb{R}^{n_x} \rightarrow \mathbb{R}$ are known smooth functions which represent the path cost function and terminal cost function.
	\begin{remark}
		It should be underscored that the specific structure of the cost function, as the sum of decoupled stage-cost terms, is not strictly necessary, although it is retained as it aligns with all available optimal control problem solvers relied upon in the design phase.
	\end{remark}
	
	In summary, the dynamic optimizing control problem can be formulated as:
	\begin{equation}
		\label{eq:dyopt}
		\begin{aligned}
			\mathcal{J}^{\star}(x_0,\boldsymbol{d}) = \min_{\boldsymbol{u}} \ & 
			 l_f(x_{L})+ \sum_{i=0}^{L-1} l_i(x_{i},u_{i}) \\
			\text{s.t.} \quad &{x}_{k+1} = f_k(x_{k},u_{k},d_{k}) \\
			& 	\boldsymbol{g}(\boldsymbol{u}, x_0, \boldsymbol{d}) \leq 0
		\end{aligned}
	\end{equation}
	$\mathcal{J}^{\star}$ denotes the optimal value of dynamic optimization problem \eqref{eq:dyopt}. 
	
	\begin{remark}
		Dynamic optimization problems can be categorized into three classes based on the nature of the objective functional $L$. When $L = \inf$, it denotes an infinite-horizon dynamic optimization problem; when $L$ is a fixed finite value, it represents a fixed-horizon dynamic optimization problem; and when L is finite yet free, it is termed an open-horizon dynamic optimization problem, which typically involves terminal conditions. Numerically solving infinite-horizon dynamic optimization problems remains a challenging endeavor. The most viable strategy is to employ a model predictive control approach, transforming the infinite-horizon problem into a sufficiently long finite-horizon dynamic optimization problem for resolution. Fixed-horizon problems can be directly solved using numerical methods. For open-horizon problems, the numerical solution approach is akin to the infinite-horizon case, whereby a sufficiently large temporal domain is considered to ensure the terminal condition is met, with the cost function set to 0 thereafter, enabling numerical resolution. Detailed numerical solution methods for dynamic programming problems can be referenced in \cite{bellman2015}.
	\end{remark}

	Without loss of generality, the constraints of this problem can also be represented in the form of the following set:
	$\boldsymbol{u} \in S_{\text{feas}}(x_0,\boldsymbol{d})$ where $S_{\text{feas}}$ denotes the feasible region of $\boldsymbol{u}$.
	
	Further, it is assumed that the dynamic optimization problem \eqref{eq:dyopt} has a (local) minimum.
	$\boldsymbol{u}^{\star} := \psi(x_0,\boldsymbol{d}) = \left[{u_0^{\star}}^{\top}  \ {u_{1}^{\star}}^{\top}  \ \dots \  {u_{L-1}^{\star}}^{\top}  \right]^{\top} \in  \mathbb{R}^{n_u  L}  $ and $\mathcal{J}^{\star}(x_0,\boldsymbol{d})$
	represent the optimal solution and optimal value of the dynamic optimization problem \eqref{eq:dyopt}. 
	Broadly speaking, the optimal solution to problem \eqref{eq:dyopt} is not unique, so represent its optimal value and optimal solutions with the following mapping can be represented:
	
	The optima value mapping acts from $\mathbb{R}^d$ to $\mathbb{R}$ and is defined by
	$$
	\mathcal{J}^{\star} : (x_0,\boldsymbol{d}) \mapsto \inf _{\boldsymbol{u}}\left\{\mathcal{J}(\boldsymbol{u} | x_0,\boldsymbol{d}) \mid \boldsymbol{u} \in S_{\text {feas }}(x_0,\boldsymbol{d})\right\} 
	$$
	The optimal value denotes the infimum of ${J}$ over the feasible set of control inputs $\boldsymbol{u}$, given the fixed initial state $x_0$ and disturbance $\boldsymbol{d}$. At the same time, it is assumed that $\inf$ is finite.
	
	The optimal set mapping is 
	defined by:
	$$
	S_{\mathrm{opt}}: (x_0,\boldsymbol{d}) \mapsto\left\{\boldsymbol{u} \in S_{\text {feas }}(x_0,\boldsymbol{d}) \mid \mathcal{J}(\boldsymbol{u} | x_0,\boldsymbol{d})=S_{\mathrm{val}}(x_0,\boldsymbol{d})\right\} .
	$$
	
	The optimal path is typically not unique, which implies that, in the general case, the optimal set mapping is not a one-to-one mapping, but rather a set-valued mapping.\cite{dontchevImplicitFunctionsSolution2014}
	
	\subsection{Implementation of the optimal operation}
	\label{sec:ipp}
	Optimizing control poses two central challenges. The first is the mathematical and computational task of solving the dynamic optimization problem to determine the optimal operating point. Various techniques can address this, typically converting large-scale dynamic optimizations into nonlinear programs for numerical solution via nonlinear programming (NLP) solvers. Relevant methods and tools are reviewed in \cite{biegler2021,bieglerAdvancesSensitivitybasedNonlinear2015,bazaraa2013}.
	
	The second issue is practical implementation of the optimal solution. While solving the optimization provides the ideal targets, additional considerations arise in real-time control. Factors like process noise, modeling errors, and disturbances must be addressed to achieve optimal or near optimal plant operation. 
	Moreover, not all states and disturbances specified in the optimization are directly measurable. 
	
	There are generally two approaches to addressing this issue. The key difference between these two approaches lies in whether measurements are utilized. If measurements are unavailable, a robust optimization framework is introduced to guarantee feasibility across various uncertainties, consequently introducing conservatism and some degree of performance loss. In contrast, the conservatism can be reduced through leveraging measurements. Measurements can be employed to accommodate uncertainties by adjusting input trajectories while considering both parametric uncertainties and disturbances.
	
	In this paper, $\boldsymbol{y}_{[k-M,k-1]}$ denotes the past measurement trajectory where $M$ denotes the observation window. The control inputs can also be regarded as generalized measured variables. Here, ${\xi}(k) := [{y}(k)^{\top},{u}(k)^{\top}]^{\top} $ and $\boldsymbol{\xi}_{[k-M,k-1]} $ denote as the extended measurements and the past extended measurement trajectory.

	Measurement-based methods can be categorized into two classes: one involves estimating the system's states and disturbances through measurements, while the other directly utilizes the measurements for feedback control.
	In the former one, the disturbances are assumed to change slowly, i.e. $d_k = d$ is constant over $M$ steps, or $d_k$ follows a fixed random distribution, i.e. $d_k \sim \mathcal{D}.$ Under this assumption, if the pair $(x_k, d)$ were reconstructible from the extended observation $\boldsymbol{\xi}_{[k-M,k-1]}$, one could reconstruct estimations $\hat{x}_k$ and $\hat{d}$ of these two quantities as functions of the previous measurements:
	$$\hat{x}_k(\boldsymbol{\xi}_{[k-M,k-1]}); \quad \hat{d}(\boldsymbol{\xi}_{[k-M,k-1]})$$
	Based on Bellman's principle of optimality \cite{bellman1966,bellman2015}, the optimal control input $u_k^{\star}$ could be treated as a function of $x_k$ and $d$. After states and disturbance are estimated, the optimized control input can be obtained under suitable optimization solvers. However, this assumption may not necessarily hold in practice \cite{doyleGuaranteedMarginsLQG1978}. 
	
	The latter one is named the direct feedback optimization methods \cite{Krishnamoorthy2022}. The most common of this is the Necessary Conditions of Optimality (NCO) tracking approach, provide a general framework for transforming an operating optimization problem into a control problem. This is because, at the optimal operating point, the first-order necessary optimality conditions must be satisfied. Therefore, if the gradients of the economic objective with respect to the control inputs can be measured (or estimated) online, then gradient descent can be used to iteratively update the control inputs until the NCO are met. However, this class of methods suffers from two main drawbacks: 1) The gradient estimation requires providing sufficient excitation, which in practical implementation, especially near the optimum, is challenging to satisfy\cite{bristow2006}; 2) For non-convex problems, since the gradient being zero is only a necessary, not sufficient, condition for optimality, gradient-based methods may converge to saddle points or local maxima.
	
%
%
	Unlike other control techniques, self-optimizing control focuses on the design of controlled variables, the outputs of the system, rather than the design of the controller, the inputs to the system. In process systems, the design of the control structure, which includes the selection of controlled variables, is more crucial than the design of the controller itself. However, the concept of dynamic self-optimizing control problems and the form of controlled variables in dynamic optimization problem has not been systematically discussed thus far. Henceforth, the definition for dynamic self-optimizing control problems and the form of dynamic controlled variables will be provided.

	\section{Dynamic self-optimizing control and Dynamic controlled Variables}
	\label{sec:DCV}
	\subsection{The definition of Dynamic self-optimizing control}
	Steady-state SOC is a technique targeted at steady-state optimization problems, serving as an alternative to real-time optimization. It achieves near-optimal performance by offline design of self-optimizing controlled variables, which can be maintained at a constant setpoint online without the need for setpoint updates.
	
	Dynamic self-optimizing control (DSOC) is the extension of SOC to dynamic optimization problems. The definition of DSOC is as follows: 
	\begin{definition}[Dynamic self-optimizing control]
		Dynamic self-optimizing control is when the selected controlled variables are  maintained at constant values (without loss of generality at 0) at every instant, an acceptable loss is achieved (without need for online re-optimization despite disturbances and measurement noise). 
	\end{definition}
	In other words, the key to DSOC is the design of controlled variables that satisfy the above requirements.
	
	This definition is analogous to the concept of steady-state self-optimizing control, with the sole distinction that here the controlled variables need to be maintained at a constant value over time, i.e., their entire operating trajectory is specified.
	In contrast, for traditional self-optimizing controlled variables (CVs), only their steady-state value needs to be maintained at the setpoint, which can be achieved by incorporating integral action in the controller design.
	
	In the conventional SOC framework, the controller design and CV design are separated. The controller design focuses on the system's dynamic and transient behavior, while the CV design and their setpoints typically target the steady-state behavior.
	
	To overcome this limitation, the traditional controlled variables that only consider steady-state behavior must be generalized to incorporate the system's dynamic characteristics. In this work, such controlled variables are termed Dynamic Controlled Variables (DCVs).
	
	The systematic analysis of dynamic controlled variables is a primary objective of this section.
	
	\subsection{The concept and form of DCVs}

		
	In order to seize the dynamic characteristics of the system, the CVs at time $k$ should be expressed as functions of the past extended measurement trajectory $\boldsymbol{\xi}_{[k-M,k-1]}$ where $M$ denotes the observation window. 
	 
	Based on whether the predicted future measurements and inputs are explicitly included, the form of dynamic controlled variables can be expressed in the following three ways:
	\begin{enumerate}
		\item Explicit forecasting form:
			$$c_k = h_1( \boldsymbol{\xi}_{[k-M,k-1]}, \hat{\boldsymbol{\xi}}_{[k,k+H-1]})$$
		where the CV mapping $h_1()$ explicitly depends on estimated future measurements and forecasted future inputs $\hat{\boldsymbol{\xi}}_{[k,k+H-1]}$ over a horizon $H$, in addition to past measured trajectories $\boldsymbol{\xi}_{[k-M,k-1]}$.
		\item  Implicit forecasting form:
			$$c_k = h_2(u_k, \boldsymbol{\xi}_{[k-M,k-1]})$$
		where the CV mapping $h_2()$ utilizes historical data and current control input, with any anticipation of the future implicitly embedded within the functional structure.
		\item No forecasting form:
			$$c_k = h_3(\boldsymbol{\xi}_{[k-M,k-1]})$$
		where the CV mapping $h_3()$ utilizes only historical data without future information.
	\end{enumerate}
	
	The key relationship is that controllability is a necessary condition for effective dynamic controlled variables. In the no forecasting form, CVs rely solely on past measurements, making them uncontrollable. To enable control, an additional controller, $u_k = \pi(\boldsymbol{\xi}_{[k-M,k-1]})$ , relating $u_k$ to $\boldsymbol{\xi}_{[k-M,k-1]}$ must be designed. Under the assumption that such a controller $\pi()$ exists, the no forecasting form and the implicit forecasting form can be transformed into each other. More specifically, the controller $\pi()$ is substituted into the implicit forecasting form:
	$$c_k = h_2(u_k, \boldsymbol{\xi}_{[k-M,k-1]}) = h_2( \pi(\boldsymbol{\xi}_{[k-M,k-1]}),\boldsymbol{\xi}_{[k-M,k-1]}) =  h_3(\boldsymbol{\xi}_{[k-M,k-1]}) $$

	The explicit forecasting form represents a receding horizon control where future inputs over a window are considered, but only the current input is implemented. The future predictions also depend on current measurements. Therefore, it can also be reduced to the implicit forecasting form. More specifically, the prediction model $\hat{\boldsymbol{\xi}}_{[k,k+H-1]} = F(\boldsymbol{\xi}_{[k-M,k-1]},u_k)$ is substituted into the explicit forecasting form:
	$$ \begin{aligned}
		c_k &= h_1(\boldsymbol{\xi}_{[k-M,k-1]}, \hat{\boldsymbol{\xi}}_{[k,k+H-1]}) \\
		&= h_1(\boldsymbol{\xi}_{[k-M,k-1]}, F(\boldsymbol{\xi}_{[k-M,k-1]},u_k)) \\
		&= h_2(u_k, \boldsymbol{\xi}_{[k-M,k-1]})
	\end{aligned}
	$$
	The relationship between these three forms is summarized in Fig~\ref{fig:DCV}.
	\begin{figure}[h]
		\centering
		\includegraphics[width=0.4\textwidth]{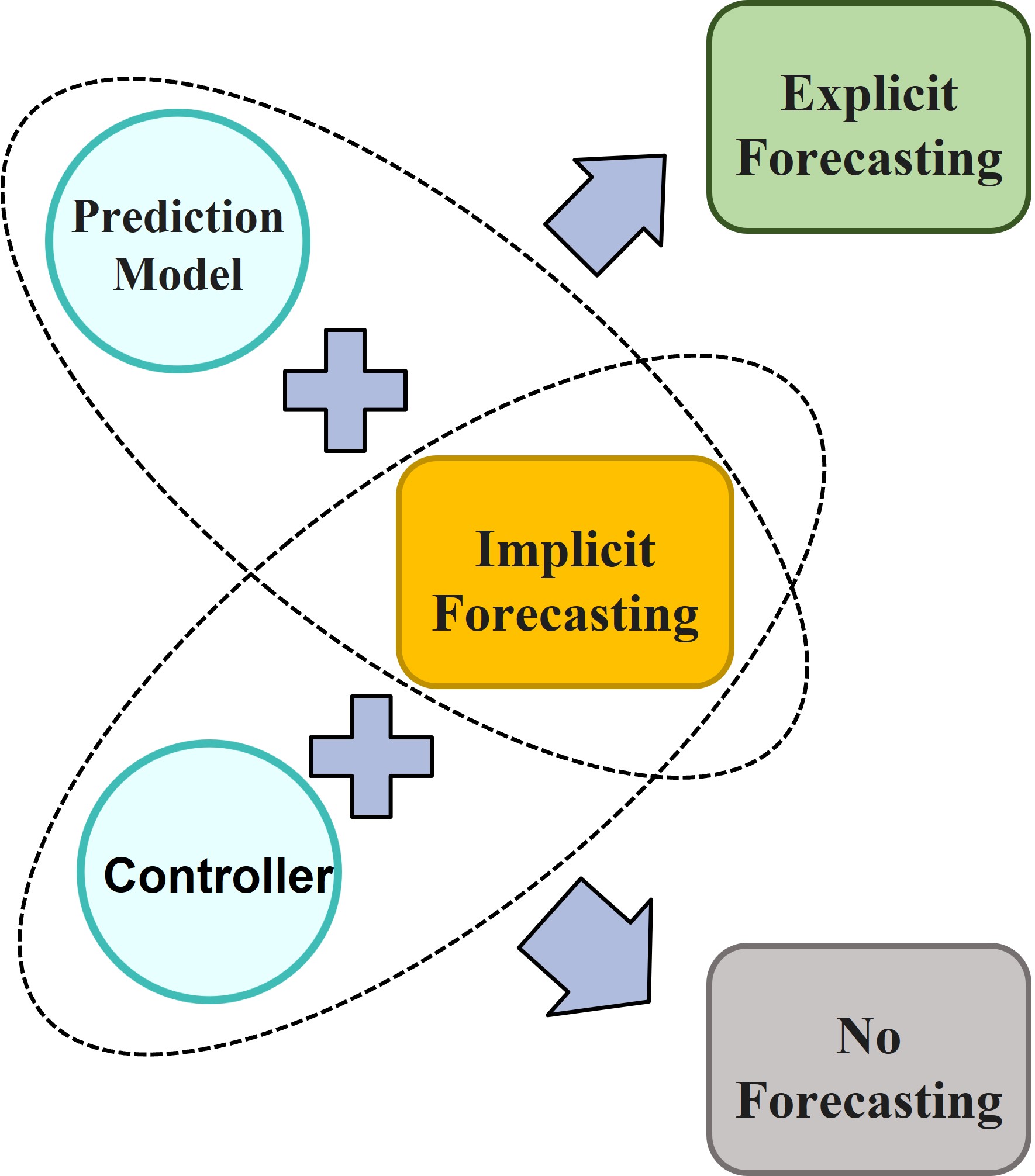}
		\caption{Relationship between these three forms of DCV}
		\label{fig:DCV}
	\end{figure}

	For ease of exposition going forward, even though numerous formulations exist for dynamic controlled variables, herein the representations under a standard form denoted is united:
		$
			c_k = h(u_k,\boldsymbol{\xi}_{[k-M,k-1]}).
		$
	Any form of DCVs can be treated as a special case of this standard formulation.
	
	\subsection{Online implementation of DCV: an implicit control policy based on DCVs}
	Next, the online implementation of dynamic self-optimizing control will be given.
	Here, an implicit control policy based on DCVs is proposed , and a theorem is given to show the generality of DCVs.
	The DCVs based policy is formulated as follows:
	\begin{equation}
		u_k = \arg_{u_k} h(u_k, \boldsymbol{\xi}_{[k-M,k-1]}) = 0
	\end{equation}
	
	As shown in Fig.~\ref{fig:DCVonline}, the online implementation of DCV encompasses a buffer stack to store extended measurements of horizon length $M$, coupled with an online equation solver that resolves the DCV system of equations in real-time, wherein the roots of the system of equations constitute the control inputs at the given instant.
	\begin{figure}[h]
		\centering
		\includegraphics[width=0.6\textwidth]{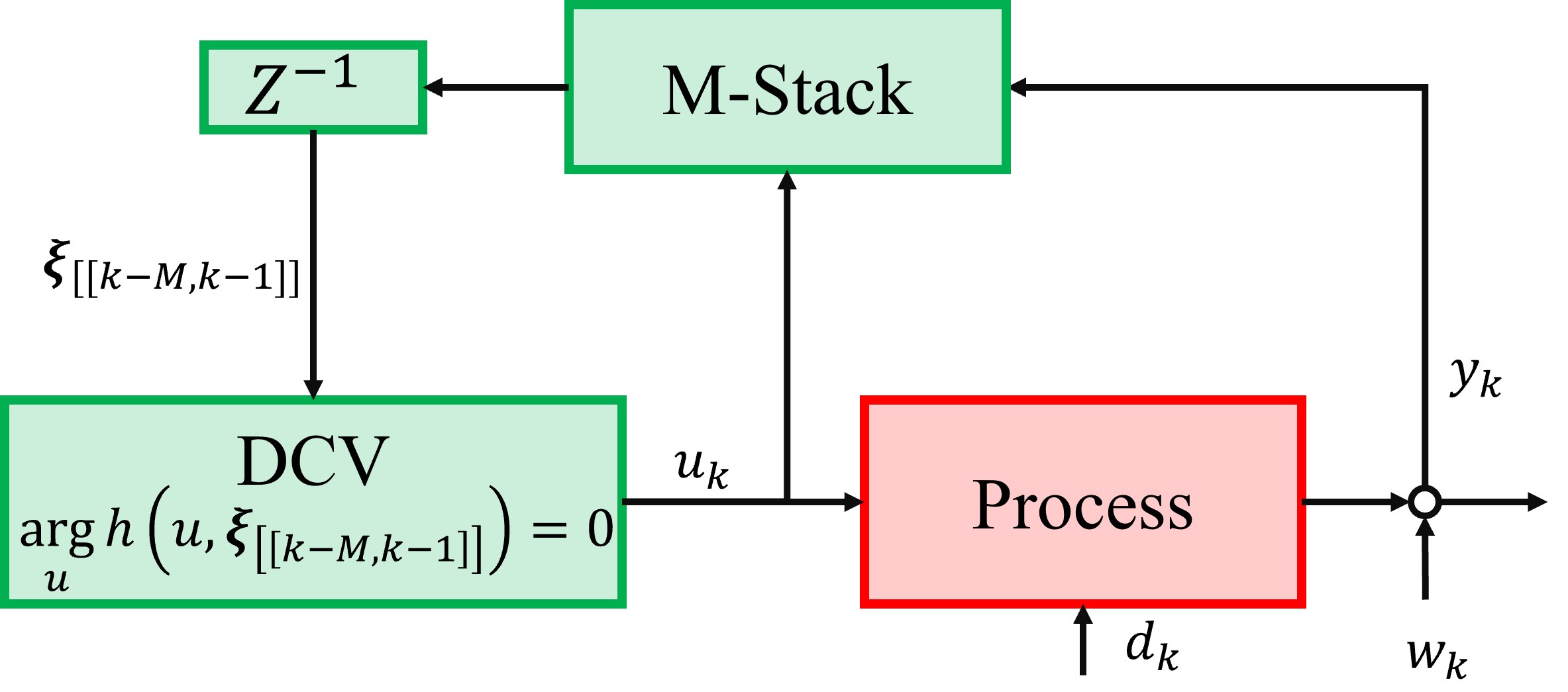}
		\caption{Online implementation structure of DCV}
		\label{fig:DCVonline}
	\end{figure}
	
	\begin{theorem}
		\label{the:DCVF}
		Let $F:\mathbb{R}^m \rightrightarrows \mathcal{Y} \subseteq \mathbb{R}^n $ be a set-valued mapping with closed graph, defined as
		\begin{align*}
			\text{gph }F = \{(\mathbf{x}, \mathbf{y}) \in \mathbb{R}^m \times \mathcal{Y} \mid \mathbf{y} \in F(\mathbf{x})\}
		\end{align*}
		and $F(\mathbf{x}) \neq \emptyset$ for all $\mathbf{x} \in \mathbb{R}^m$. Then there exists a continuous function $g:\mathbb{R}^m \times \mathbb{R}^n \rightarrow \mathbb{R}^n$ such that the solution mapping
		\begin{align*}
			S:\mathbb{R}^m \rightrightarrows \mathcal{Y}, \quad S(\mathbf{x}) = \{ \mathbf{y} \in \mathbb{R}^n \mid g(\mathbf{x},\mathbf{y}) = 0 \} 
		\end{align*}
		satisfies $S(\mathbf{x}) \neq \emptyset$ for all $\mathbf{x} \in \mathbb{R}^m$ and
		\begin{align*}
			F(\mathbf{x}) = S(\mathbf{x}), \quad \forall \mathbf{x} \in \mathbb{R}^m.
		\end{align*}
	\end{theorem}
	\begin{proof}
		See appendix.
	\end{proof}
	
	Based on Theorem~\ref{the:DCVF}, a large set of functions can be represented by DCVs, such as multi-valued functions, discontinuous functions, etc. Relative to explicit functions, DCVs possess advantages in this implicit formulation.
	While it cannot be said that implicit functions can represent more functions than explicit functions, explicit functions require a prior knowledge of the location of discontinuities or multi-values to represent such functions, whereas implicit functions do not.	
	This will also be demonstrated in the case studies.

	\subsection{Relationship between DCV and controller}
	\label{sec:DCV&controller}
	In current industrial practice and static SOC \cite{skogestad2004,halvorsen2003}, the commonly used CV is a special case of the no forecasting form, utilizing only the current measurement as the CV,
	$$c_k = E y_k , $$ where $E \in \mathbb{R}^{n_u \times n_y}$ denotes a linear combination matrix with $rank(E) = n_u$. In this form of DCV, an additional controller is usually required to maintain the DCV at the setpoint.
	
	Nearly all existing controller design methods can be considered special cases of the DCV forms. For example, a feedback control policy denoted by $u_k = \pi(\boldsymbol{\xi}_{[k-M,k-1]})$:
	\begin{align*}
		c_k &= u_k - \pi(\boldsymbol{\xi}_{[k-M,k-1]})\\
		&= h(u_k, \boldsymbol{\xi}_{[k-M,k-1]}) \\
		c_k=0 \Longrightarrow u_k = \pi(\boldsymbol{\xi}_{[k-M,k-1]})
	\end{align*}
	where $\pi$ denotes the feedback control policy.
	
	Similarly, receding horizon control or model predictive control, whose control policy can be denoted as $\boldsymbol{u}_{[k,k+H]} = \pi(\boldsymbol{\xi}_{[k-M,k-1]}, \hat{\boldsymbol{\xi}}_{[k,k+H-1]})$, with only $u_k$ actually implemented, can also be expressed as a special DCV form:
	\begin{align*}
		c_k &= h(\boldsymbol{\xi}_{[k-M,k-1]}, \hat{\boldsymbol{\xi}}_{[k,k+H-1]}) \\
		&= \boldsymbol{u}_{[k,k+H]} - \pi(\boldsymbol{\xi}_{[k-M,k-1]}, \hat{\boldsymbol{\xi}}_{[k,k+H-1]}) \\
		&\Longrightarrow \boldsymbol{u}_{[k,k+H]} = \pi(\boldsymbol{\xi}_{[k-M,k-1]}, \hat{\boldsymbol{\xi}}_{[k,k+H-1]})
	\end{align*}
	
	It can be seen from the above analysis that when the value of DCVs is specified, it can be converted with most current controller design methods. It can also be said that in the explicit forecasting form and implicit forecasting form of DCVs, the controller is implicit in the controlled variable.
	From these perspectives, the relationship between the controller and DCV is very close.	
	To understand the relationship between DCVs and controllers, it is first necessary to understand their respective concepts and functionalities.
	
	\textbf{Concepts}
	\begin{enumerate}
		\item 	Dynamic controlled Variables (DCV): A variable introduced into a dynamic system to limit its degrees of freedom and influence its behavior. It acts as an additional constraint, guiding the system towards a desired state.
		\item  Controller: An algorithm that receives system measurements and generates control inputs to manipulate the system's behavior and achieve a desired outcome.
	\end{enumerate}
	The definitions of DCVs and controllers provide a high-level overview of their key features. However, to fully understand their relationship, it is necessary to examine their functionalities in more detail.
	
	\textbf{Functionalities}
	
	The functionalities of DCVs and controllers provide a more detailed understanding of how they operate. By comparing and contrasting their functionalities, it will be seen clearly how they interact and influence each other.
	
	DCV:
	\begin{enumerate}
		\item Proactive: Shapes the system's trajectory by anticipating future states and incorporating them into control decisions.
		\item Constraint: Provides a reference output for the system, influencing its behavior without directly manipulating the system.
		\item Adaptable: Can be used with various control strategies by adjusting its definition based on the desired behavior.
	\end{enumerate}
	
	Controller:
	\begin{enumerate}
		\item Reactive: Adjusts the system's behavior based on current measurements and deviations from the desired state.
		\item Direct Manipulation: Directly affects the system's behavior by manipulating its inputs.
		\item Specific: Tailored to a specific control strategy and may not be readily adaptable to others.
	\end{enumerate}
	
	The functionalities of DCVs and controllers provide a clear contrast between the two concepts. DCVs are proactive, shaping the system's behavior by anticipating system output. Controllers are reactive, adjusting the system's behavior based on current measurements. This contrast is evident in the following table:
	\begin{table}[h]
		\caption{The difference between DCV and controller}
		\label{tab:DCV&controllrrt}
		\centering
		\begin{tabular}{lll}
			\hline
			\textbf{Feature} & \textbf{Dynamic controlled Variables} & \textbf{Controller} \\ \hline
			Function & Shapes system behavior & Manipulates system inputs \\
			Information used & Past measurements, future predictions & Past measurements \\
			Output & The value of DCV & Control inputs for the system \\
			Role & Proactive guidance & Reactive adjustment \\
			Flexibility & Adapts to different control strategies & Specific to chosen strategy \\ \hline
		\end{tabular}
	\end{table}
	
	The concept of dynamic controlled variables is emerged as a novel approach to control theory, offering an alternative perspective on shaping the behavior of dynamic systems. While traditional control methods rely on reactive adjustments based on current measurements, DCVs introduce a proactive element by incorporating predictions of future output into the control loop. This proactive approach opens up new possibilities for achieving efficient, precise, and predictable control.
	
	\subsection{Relationship between DSOC and Sliding mode control}
	Sliding mode control (SMC) is a nonlinear control technique that belongs to the broader category of variable structure control systems. The key idea in SMC is to drive the system state trajectory onto a prescribed surface in the state space, known as the sliding surface, and then maintain the state trajectory on this surface. This is achieved by means of a discontinuous control law that switches the control structure.
	
	The sliding surface is a manifold to which, when the system trajectory is constrained in this manifold, the system exhibits the desired behavior. 
	The sliding surface is typically defined as a hyperplane of the state variables, or a specified nonlinear function of specified state variables. Some research has also extended the sliding surface to include integral functions of the state variables. In terms of form, the nonlinear structure of the sliding surface is limited to some manually designed forms, without a generalized approach for designing nonlinear sliding surfaces.
	Furthermore, DCVs explicitly incorporate predicted variables and control inputs, which the sliding surface does not include. The integral terms in the sliding surface can be considered as representing past trajectories, and thus the sliding surface can be viewed as DCVs with no forecasting form.
	
	The sliding surface merely describes the ideal region, without specifying requirements on how to reach it or how to maintain the system within it. 	
	To implement sliding mode control, an additional controller needs to be designed to ensure the system trajectory intersects with the sliding surface and ultimately remains confined to it. The key distinction is that DCVs inherently contain an implicit control law within their formulation, without requiring additional explicit controller design. This means that DCVs describe the ideal manifold accessible from any initial state, whereas the sliding surface only captures the desired target manifold to be reached.
	
	The differences between sliding mode control and dynamic self-optimization are finally summarized in Table \ref{tab:SMC&DSOC}.
	\begin{table}[ht]
		\caption{The differences between sliding mode control and dynamic self-optimizing control}
		\centering
		\label{tab:SMC&DSOC}
		\begin{tabular}{lll}
			\hline
			 \textbf{Aspect} &  \textbf{Sliding Mode Control}&  \textbf{Dynamic SOC}\\ \hline
			 Control Law&
			 \begin{tabular}[c]{@{}l@{}}Explicitly designed control law \\ to drive system to sliding surface\end{tabular}&
			 \begin{tabular}[c]{@{}l@{}}Implicit control law embedded \\ in DCVs\end{tabular}\\
			 Desired Behavior&  Confined to sliding surface    &  Maintaining DCVs constant                 \\
			 Manifold Description&
			 \begin{tabular}[c]{@{}l@{}}Sliding surface specifies desired \\ target manifold\end{tabular}&
			 \begin{tabular}[c]{@{}l@{}}DCVs describe near-optimal \\ manifold from any initial state\end{tabular}\\
			 Transient Response&
			 Reaching phase and sliding phase& No explicit separation
			 \\ \hline
		\end{tabular}
	\end{table}
	In summary, dynamic SOC originates from the perspective of control structure design, whereas SMC does not consider control structure design. By approaching from the control structure design perspective, dynamic SOC introduces novel ideas that are more comprehensive. This also necessitates a holistic consideration of the system's dynamic behavior in dynamic SOC.

	\section{Self-optimizing DCVs design}
	\label{sec:DSOC}
%
	
	\subsection{The requirements of self-optimizing DCVs}\label{sec:req}
	In closed-loop control, the CVs $c_k$ are maintained at the constant set-points $0$. However, operation may be non-optimal with a positive loss due to two types of errors:
	\begin{enumerate}
		\item Setpoint error ($e_{cs,k}$): The difference between the setpoint $\hat{c}_k$ and the truly optimal $c_k$ value ($c_{opt,k}$) under current conditions at time $k$. This stems from changes in operating conditions and the system dynamic.
		\item Implementation error ($e_{tc,k}$): The difference between the actual $c_k$ value and its setpoint $0$. This results from measurement noise, control performance limitations, etc.
	\end{enumerate}
	The overall error is the sum:
	\begin{equation}
		e_{c,k} = e_{cs,k} + e_{tc,k}
	\end{equation}
	This overall error causes suboptimal operation and economic loss.
	
	From this, Skogestad and Postlethwaite \cite{skogestad2005,skogestad2000} state some requirements for good CVs:
	\begin{enumerate}
		\item The CV should be easy to control, that is, the inputs $u$ should have a significant effect (gain) on $c$.
		\item The optimal value of $c$ should be insensitive to disturbances.
		\item The CV should be insensitive to noise.
		\item In case of several CVs, the variables should not be closely correlated.
	\end{enumerate}
	These requirements are posed for static optimization problems, and tend to be an engineering definition, not extremely rigorous. Next, the requirements for dynamic controlled variables will be given from the perspective of reducing both setpoint error and implementation error.
	
	\begin{requirement}
		The CV's optimal value $c_{opt,k}$ should be insensitive to disturbances and operating changes. In other words, it should be time invariant.
	\end{requirement}
	A good CV's optimal value should remain fixed over time and unaffected by both internal and external changes.
	A setpoint error index $\alpha$ is defined to quantify the expected distance between the optimal value of DCVs and selected fixed setpoints:
	\begin{equation}
		\alpha = \underset{\boldsymbol{d},\boldsymbol{w},{x}_0}{\mathbb{E}} \left\{ \| c_{opt,k}-\hat{c}_k \| \right\} = \underset{k}{\mathbb{E}} \left\{ \| c_{opt,k}-\hat{c}_k \| \right\}
	\end{equation}
	As such, $\alpha$ provides a metric to guide selection and design of dynamic controlled variables. Comparing CV choices based on their $\alpha$ values reveals which variables exhibit the greatest invariance.	

	\begin{requirement}
		The DCV should be insensitive to noise.
	\end{requirement}
	Let the DCV mapping be defined as $c_k = h(u_k, \boldsymbol{\xi}_{[k-M,k-1]})$, where $u_k$ and $\boldsymbol{\xi}_{[k-M,k-1]}$ contain measurement noise. Let $u_{g,k}$ and $\boldsymbol{\xi}_{g,[k-M,k-1]}$ denote the noiseless perfect measurements. Then the noise insensitivity requirement can be formalized as:
	\begin{equation}
		\beta =  \underset{\boldsymbol{w}}{\text{var}}  \left\{h(u_k,\boldsymbol{\xi}_{[k-M,k-1]}) - h(u_{g,k},\boldsymbol{\xi}_{g,[k-M,k-1]})\right\}
	\end{equation}
	If $\beta$ is sufficiently small, this ensures the deviations in the controlled variable can be constrained within a sufficiently small range despite noise in the measurements used to compute it. Thus the CV can remain immune to typical measurement errors and fluctuations, providing a steady reference signal.
	
	\begin{requirement}
		The DCV should be easy to control accurately. 
	\end{requirement}
	In order to qualify this requirement, two indices are introduced to measure the “difficulty” of achieving accurate control for DCV.
	Before giving the definition of two indices, the Jacobian matrix $J({\ell})$ is introduced:
	$$J({\ell}) = \frac{\partial c_k}{\partial \boldsymbol{u}_{[k-\ell,k]}}$$
	and the rearranged Jacobian matrix $J_r({\ell})$:
	$$ J_r({\ell}) =\left[
	\begin{array}{c}
		\frac{\partial c_k}{\partial{u}_{k}} \\
		\frac{\partial c_k}{\partial{u}_{k-1}} \\
		\vdots \\
		\frac{\partial c_k}{\partial{u}_{k-\ell}} 
	\end{array} 
	\right] $$

	First is the DCV index $\ell$ which measure the distance from a CV to its related control input.
	The index is a nonnegative integer that provides useful information about the mathematical structure and potential complications in the analysis and the implementation issue of DCV.
	\begin{definition} [Index $\ell$ of DCV]
		The index $\ell$ of a dynamic controlled variable $c_k$ is defined as the minimum non-negative integer $\ell$ such that the Jacobian matrix:
		$J({\ell}) $
		has full row rank. That is:		
		$$\text{rank}(J({\ell})) = \text{dim}(c_k) = n_u$$
		and the rearranged Jacobian matrix 
		$ J_r({\ell})$
		has full column rank. That is:		
		$$\text{rank}(J_r({\ell})) = \text{dim}(u_k) = n_u$$
	\end{definition}
	
	The index $\ell$ conveys a key structural property - specifically, the shortest pathway for inputs to independently influence the dynamic controlled variables. By taking the partial derivative of current CVs $c_k$ with respect to windows of prior inputs $u_{[k-\ell,k]}$, the Jacobians reveal directions of sensitivity and coupling.
	The full row rank condition of $J({\ell}) $ implies every CV element is affected by at least some input. This is the reachability requirement - no CV decoupled from all inputs.
	Meanwhile, the full column rank condition of $J_r({\ell})$ ensure each input dimension impacts some CV. This is the independence need - the ability to actuate CVs along any input direction.
	Together, the index quantifies a controllability notion - the minimum temporal distance an input change takes to observe its effect. Low index signify tighter input-CV connections. An index of 0 would mean current CVs directly depend on the current input. This index also conveys system inertia or delay - higher values mean longer lead time for inputs to impact outputs.
	
	By distilling dynamic coupling into a single non-negative integer, the index provides invaluable information about the inherent structure relating inputs and controlled variables. This facilitates analysis and design of control schemes to achieve desired CV behavior.
	
	
	A large index $\ell$ makes DCV control more challenging, similar to controlling systems with long time delays.
	A DCV with a large $\ell$ implies anticipating and accounting for uncertainties further into the future when determining the current input $u_k$. This higher dimensionality of future uncertainties poses difficulties in balancing their effects to achieve proper DCV tracking over time.
	In summary, a large DCV index $\ell$ corresponds to long control delays, complicating the selection of $u_k$ to maintain good performance. A small $\ell$ is thus preferred for tractable DCV control.
	
	The second is a gain index $\gamma$ of DCV:
	\begin{definition}
		The gain index $\gamma$ of DCV is defined as the minimum singular value of rearranged Jacobian matrix $J_r({\ell})$:
		\begin{equation}
			\gamma = \min \sigma( J_r({\ell})) 
		\end{equation}
		Where $J_r({\ell})$ is the rearranged Jacobian matrix from the CV index definition, and $\sigma( J_r({\ell}))$ denotes its singular values. 
	\end{definition}
	A higher minimum singular value implies higher gain - each input direction strongly influences the controlled variables.
	This helps ensure that control calculations are well-conditioned with substantial actuation in all dimensions. Furthermore, a larger gain index $\gamma$ reinforces noise attenuation through feedback control. By closing loops with high $\gamma$, variability in CVs gets suppressed, tightening distributions about setpoints.
	
	In summary, a large $\gamma$ index signifies both sensitive actuation and enhanced noise rejection. Together with a small CV index $\ell$ for direct feedthrough, the control scheme can rapidly compensate deviations. This allows accurately regulating the CV trajectory with minimal implementation errors. 

	

	\subsection{DCV design viewed as a map identification problem}
	In current self-optimizing controlled variable design methods, global self-optimizing control approaches \cite{Ye2015} start from optimal data and attempt to identify self-optimizing controlled variables from the data. Ye \cite{ye2023} discusses the global SOC method from a more general perspective, proposing g2SOC which views controlled variable design as a map identification problem. Inspired by this idea, this paper also views DCV design as a map identification problem. 
	So far, most self-optimizing control methods discuss controlled variables that are linear combinations of measurements. Recently, some articles discuss nonlinear controlled variables, for example, Ye \cite{ye2023} and Su \cite{su2022a} discuss the need to introduce nonlinear controlled variables in the presence of active constraint switching, but these discussions are still under the steady-state optimization framework. In dynamic optimization problems, switching between different active constraints is very common. And for problems like batch process control, there is no steady-state and the process model is usually highly nonlinear and hard to approximate with a linear model. 
	Therefore, DCV is necessary to be expressed in the nonlinear form. In this paper, general function approximators, such as a deep ReLU-MLP network, are used to parameterize the DCV:
	\begin{equation}
		c_k = h_{\theta}(u_k,\boldsymbol{\xi}_{[k-M,k-1]})
	\end{equation}
	where $h_{\theta}()$ represents the function approximator with parameters $\theta$.
	
	{Not only that, but current self-optimizing control methods for dynamic optimization are tailored to specific dynamic optimization problems, such as batch process optimization with fixed time horizons\cite{ye2018}. They are unable to address dynamic optimization problems with non-fixed time horizons. In this paper, a framework similar to approximate optimal control \cite{alamir2021} is subsequently adopted to design dynamic controlled variables, overcoming the limitations of existing methods in this regard.}
			
	Having formulated self-optimizing DCV design as a map identification problem, two key questions follow: How to construct a dataset and how to train the DCV. The following subsections will address these two important questions for enabling data-driven DCV design through map identification.
	\subsubsection{Dataset Generation}
	In Section \ref{sec:ipp}, two classes of measurement-based optimizing control approaches were discussed, distinguished by whether an explicit process model is utilized. The DCV formulation can accommodate both approaches. Inspired by related work \cite{alamir2021}, here is a proposed algorithm to construct a learning dataset to fit the map $h$:
	
	\begin{algorithm}[ht]
		\caption{Training Dataset Generation}
		\label{alg:DCV_data}
		\KwIn{ Optimization problem $\mathcal{P}$ \eqref{eq:dyopt}, system model $\mathcal{M}$ \eqref{eq:sys}, \eqref{eq:sys_m}, Predictive Model $\mathcal{\hat{M}}$ \\
			Observation horizon $M$, Number of samples $N$,  Prediction Horizon $H$, Nominal disturbance $\boldsymbol{d}_0$}
		\KwOut{Dataset $\mathfrak{D}$}
		\BlankLine
		\ForEach{$i $ in  $1$ \KwTo $N$}{
			Randomly sample $x_0 \in \mathcal{X}$ and $\boldsymbol{d} \in \mathcal{D}^{L}$, 
			$q \leftarrow (x_0, \boldsymbol{d})$
			
			$\boldsymbol{u}^\star \leftarrow \text{Solve } \mathcal{P}(q) \text{ using NLP solver (e.g. IPOPT-Casadi \cite{andersson2019})}$
			
			Randomly sample $\boldsymbol{w} \in \mathcal{W}^{L} $, 
			
			$\boldsymbol{\xi}^\star \leftarrow \text{Simulate } \mathcal{M}(q, \boldsymbol{u}^\star(q),\boldsymbol{w}) $
			
			\ForEach{$k $ in $ M$ \KwTo $L-H$}{
				$\boldsymbol{\hat{\xi}}^\star \leftarrow \emptyset$
				
				\If{$H>0$}
				{
					$q_0\leftarrow (x_0, \boldsymbol{d}_0)$ , $\boldsymbol{w} \leftarrow 0$
					
					$\boldsymbol{\hat{\xi}}^\star_{[k,k+H-1]}) \leftarrow \text{Simulate }  \mathcal{\hat{M}}(\boldsymbol{\xi}^\star_{[k-M,k-1]},\boldsymbol{u}_{[k,k+H-1]}^\star)$
					
				}
				
				$\mathfrak{D} \leftarrow \mathfrak{D} \cup {({u}_k^\star, \boldsymbol{\xi}^\star_{[k-M,k-1]}, \boldsymbol{\hat{\xi}}^\star_{[k,k+H-1]})}$
			}

		}
		\Return{Dataset $\mathfrak{D}$}
	\end{algorithm}
	
	In Algorithm \ref{alg:DCV_data}, when solving the optimal control problem, all uncertainties are assumed to be perfectly known, so the obtained solution $\boldsymbol{u}^\star(q)$ is the ideal optimal solution.
	
	\begin{figure}[h]
		\centering
		\includegraphics[width=0.6\textwidth]{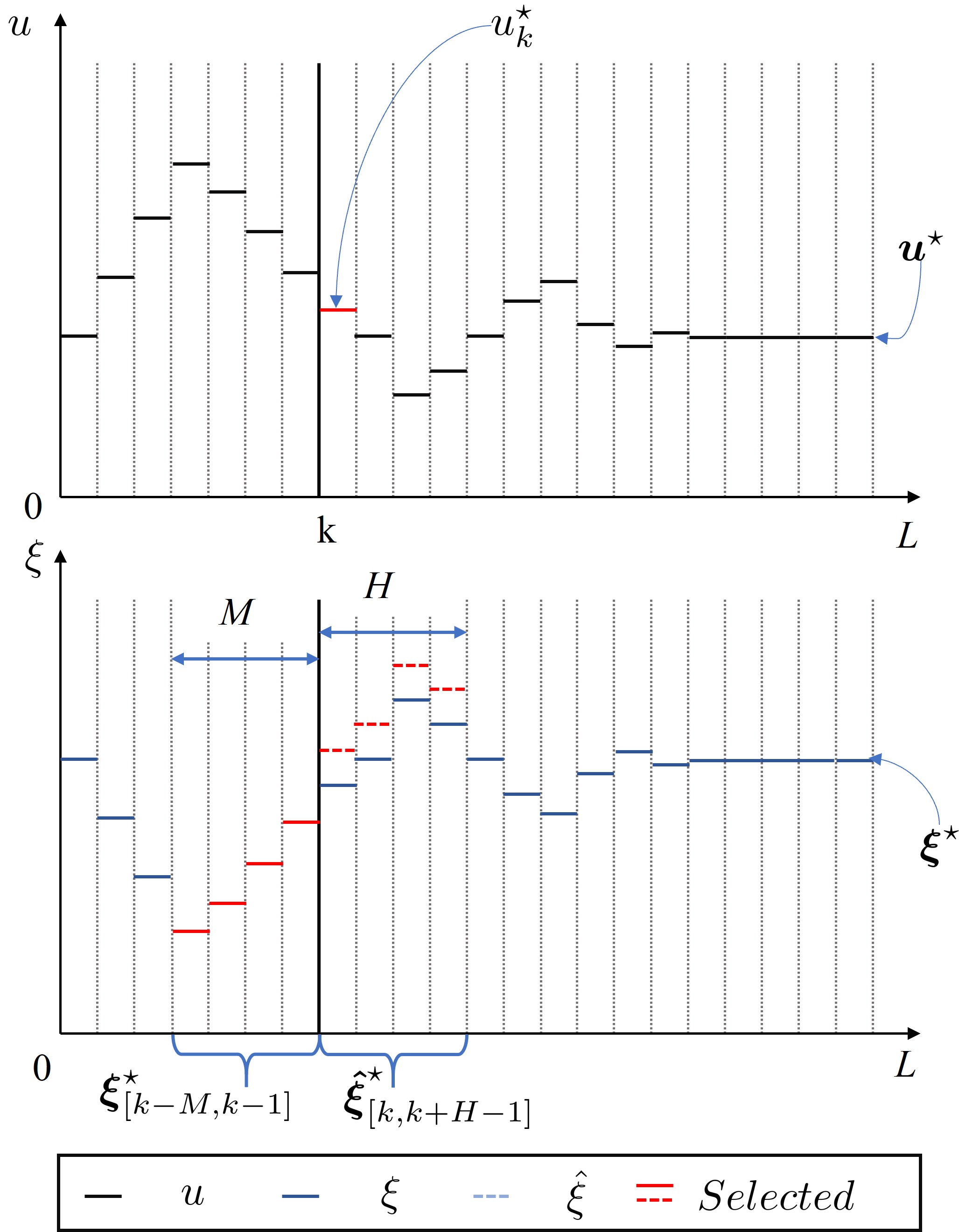}
		\caption{Composition of a sample }
		\label{fig:data_sample}
	\end{figure}
	Figure \ref{fig:data_sample} illustrates the composition of each data sample, consisting of the optimal control input ${u}_k^\star$ paired with the extended observations $\boldsymbol{\xi}_{[k-L,k-1]}^\star$ and estimated future observations $\boldsymbol{\hat{\xi}}_{[k,k+H-1]}^\star$. 

	\subsubsection{DCV model training}	
	The four indices defined in Section~\ref{sec:req} to formalize the three requirements for good DCV pose a multi-objective optimization challenge with complex trade-offs in simultaneously satisfying all the requirements. Rather than hand-tuning selector weights, a consolidated approach based on the following insight is proposed:
	For small variations around the optimal control input $u^{\star}_k$,  the relationship between $u_k$ and any candidate controlled variable set $c_k$ can be linearized as:
	\begin{equation}
		c_k - c^{\star}_k = G_k   (u_k - u^{\star}_k)
	\end{equation} where $G_k = \frac{\d c_k}{\d u_k}\bigg|_{u_k = u^{\star}_k}$ and $c_{k}^{\star} = h(u_{k}^{\star},\boldsymbol{\xi}^\star_{[k-M,k-1]})$.
	Rearranging gives:
	\begin{equation}
		u_k - u^{\star}_k = G_k^{-1} (c_k - c^{\star}_k)
	\end{equation}
	Since $(u_k - u^{\star}_k)$ is expected to be as small as possible, it follows that selecting $c_k$ to minimize the product of $\left(\frac{\d c_k}{\d u_k}\right)^{-1}$ and $(c_k - c^{\star}_k)$. 
	An enough small $\left(\frac{\d c_k}{\d u_k}\right)^{-1} (c_k - c_k^\star)$ can meet the three requirements because:
	\begin{enumerate}
		\item Small norm of $c_k - c_k^\star = (c_k-c_{g,k})+(c_{g,k}- c_k^\star)$ implies: 
		\subitem(a) $c_{g,k} - c_k^{\star}$ has small norm, representing the small distance between no-noisy values of controlled variables to $c_k$, satisfying Requirement 1
		\subitem(b) $(c_k-c_{g,k})$ has small norm, representing the small change of $c_k$ under noise,  satisfying Requirement 2.
		\item Small $G_k^{-1}$ implies:
		\subitem(a) $G_k$ is invertible, meaning the CV index $\ell$ is 0 to meet Requirement 3.
		\subitem(b) $G_k^{-1}$ has small norm, i.e. large norm of $G_k$, meaning gain index $\gamma$ is large, also satisfying Requirement 3.
	\end{enumerate}
	
	In the face of uncertainties, the performance of DCVs necessitates a more nuanced evaluation. One approach leverages the worst-case loss, while another harnesses the statistical properties of the loss distribution, such as the expected loss over all uncertainties.
	In this paper, without loss of generality, further assuming the controlled variable is consistently maintained at $0$ in closed-loop, i.e. $c_k = 0$, then the following loss function could be employed to evaluate the DCV's performance
	\begin{equation}
		\label{eq:dsocLoss0}
		\mathcal{L} = \underset{\mathcal{X}_0,\mathcal{D},\mathcal{W}}{\mathbb{E}} \left\{ \sum_{i=0}^{L-1} \left\|G_k^{-1} c_k^\star \right\|\right\} 	
	\end{equation}
	\begin{remark}
		If the DCV is the explicit forecasting form, 
		$$G_k = \frac{\d c_k}{\d u_k}\bigg|_{u_k = u^{\star}_k} = \frac{\partial c_k}{\partial u_k}\bigg|_{u_k = u^{\star}_k} + \frac{\partial c_k}{\partial \boldsymbol{\hat{\xi}}_{[k,k+H-1]}}\frac{\partial  \boldsymbol{\hat{\xi}}_{[k,k+H-1]}}{\partial u_k}\bigg|_{u_k = u^{\star}_k}$$
	\end{remark}
	So the self-optimizing DCVs training problem can be summary as follows:
	\begin{equation}
		\begin{aligned}
			\min_{\theta} \mathcal{L} =\underset{\mathcal{X}_0,\mathcal{D},\mathcal{W}}{\mathbb{E}} \left\{ \sum_{i=0}^{L-1} \left\|G_k^{-1} c_k^\star \right\|\right\}\\
			\text{s.t.} c_k = h_{\theta}(u_k,\boldsymbol{\xi}_{[k-M,k-1]})
		\end{aligned}
	\end{equation}
	Usually, the Monte Carlo sampling is used to estimate $\mathcal{L} $
	\begin{equation}
	\label{eq:dsocLoss}
	\mathcal{L} = \underset{\mathcal{X}_0,\mathcal{D},\mathcal{W}}{\mathbb{E}} \left\{ \sum_{i=0}^{L-1} \left\|G_k^{-1} c_k^\star \right\|\right\} 	 \approx \frac{1}{N} \sum_{i=1}^{N} \sum_{k=0}^{L-1} \left\|G_{i,k}^{-1}  c_{i,k}^\star \right\|
	\end{equation}
	where $c_{i,k}^\star$ and $G_{i,k}^{-1}$ are the optimal DCV value and gain matrix at time instant $k$ under $x_{i,0}$, $\boldsymbol{d}_i$ and $\boldsymbol{w}_i$ where are Monte Carlo samples over $\mathcal{X}_0$, $\mathcal{D}$, $\mathcal{W}$, and $N$ is the number of samples.
	
	In the above, matrix inversion is involved, which can make DCV training difficult and computationally expensive for large datasets or high-dimensional control inputs. Also, since it only contains information about the optimal point and vicinity, the controlled variable values near non-optimal inputs are not constrained. This can lead to points where the DCV is 0 arising in these regions. During online inference, this may cause the DCV policy to produce not optimal control inputs. 
	Motivated by this, the following improved loss function is proposed:
	\begin{equation}
		\label{eq:dsocLoss2}
		\mathcal{L} 	 
		\approx \frac{1}{N} \sum_{i=1}^{N} \sum_{k=0}^{L-1}  \left(
		 \left\| c_{i,k}^\star \right\| 
		+  \sum_{j=1}^{N_c} \left\|\tau(u_{i,k}^\star-u_{i,k,j})-  c_{i,k,j} \right\| 
		\right)
	\end{equation}
	 where $c_{i,k,j} = h(u_j,\boldsymbol{\xi}^\star_{[k-M,k-1]})$ is the controlled variable value at disturbance $i$, time $k$ and input $u_j$. $N_c$ is the number of collocation point.
	 $\tau$ is a hyperparameter for the desired gain.
	 In this loss function, the corresponding negative samples are introduced to pair with the positive samples, in order to inject global information into the DCV function. The hyperparameter $\tau$ introduced in \eqref{eq:dsocLoss2} represents the expected Lipschitz constant of the DCV with respect to $u_k$. A larger $\tau$ implies a desire for a larger minimum singular value of $G_k^{-1}$. The negative samples are randomly sampled within a sphere centered at ${u}_{i,k}^\star$ with radius $r$, i.e. $\rho_r({u}_{i,k}^\star,r)$. In other words, the sampling is performed within a certain range around the optimal input.
	 
	 The training problem for the DCV can be reformulated as a regression problem, with the inputs being $u_k$ and $\boldsymbol{\xi}^\star_{[k-M,k-1]})$, and the output being the DCV value. So the DCV training can be carried out using regression training methods, leveraging established tool kits.

	{The self-optimizing DCV training procedure is summarized in Algorithm~\ref{alg:dcv}.}
	\begin{algorithm}
		\caption{Contrastive learning for DCV model training}\label{alg:dcv}
		\KwIn{Dataset $\mathfrak{D}$, deep neural network $h_\theta(.)$ , $\tau$ , $r$ and $N_c$}
		\KwOut{The trained deep neural network $h_{\hat{\theta}}(.)$}
		$\mathfrak{D}_{\text{neg}}\leftarrow \emptyset$
		
		$\mathfrak{D}_{\text{pos}}\leftarrow \emptyset$
		
		\ForEach{$({u}_{i,k}^\star, \boldsymbol{\xi}^\star_{i,[k-M,k-1]}$ in $\mathfrak{D}$}{
			\For{$j=1,2,…,N_c$}{
			$u_{i,k,j} \leftarrow$ Randomly sample in $\rho_r({u}_{i,k}^\star,r)$  
			$\mathfrak{D}_{\text{neg}}\leftarrow \mathfrak{D}_{\text{neg}} \cup (\tau(u_{i,k,j}-{u}_{i,k}^\star),[u_{i,k,j};\boldsymbol{\xi}^\star_{i,[k-M,k-1]}])$
			}
			$\mathfrak{D}_{\text{pos}}\leftarrow \mathfrak{D}_{\text{pos}} \cup (0,[u_{i,k,j};\boldsymbol{\xi}^\star_{i,[k-M,k-1]}])$
		}
		
		$\mathfrak{D}\leftarrow \mathfrak{D}_{\text{pos}} \cup  \mathfrak{D}_{\text{neg}}$ 	
			
		$\hat{\theta} \leftarrow \text{regression trainer}(\mathfrak{D})$
		
	\end{algorithm}
	
	\begin{figure}[ht]
		\centering
		\begin{subfigure}[b]{0.45\textwidth}
			\includegraphics[width=1\linewidth]{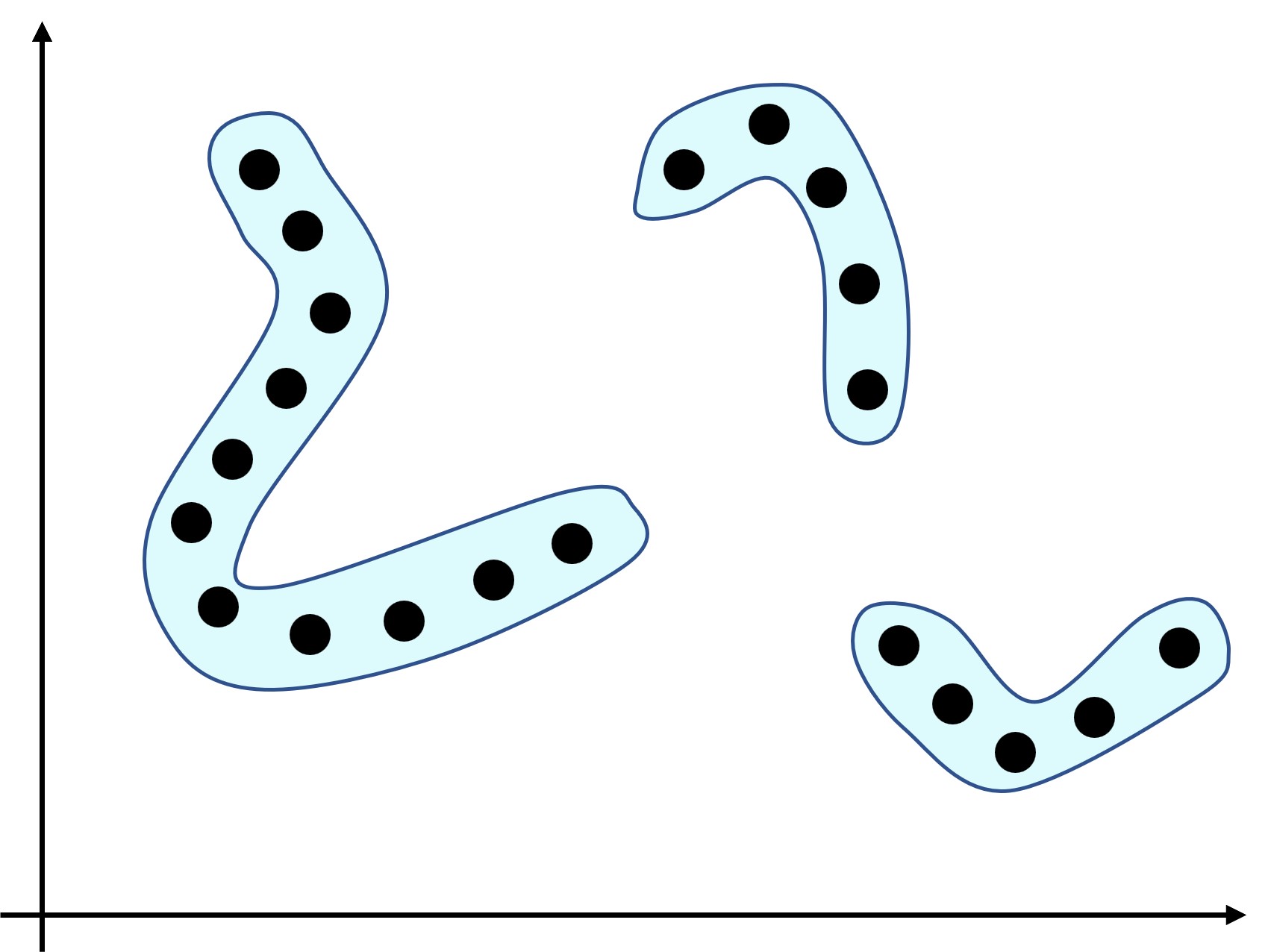}
			\caption{Local regularization learning}
			\label{fig:regLearning}
		\end{subfigure}
		\hfill
		\begin{subfigure}[b]{0.45\textwidth}
			\includegraphics[width=1\linewidth]{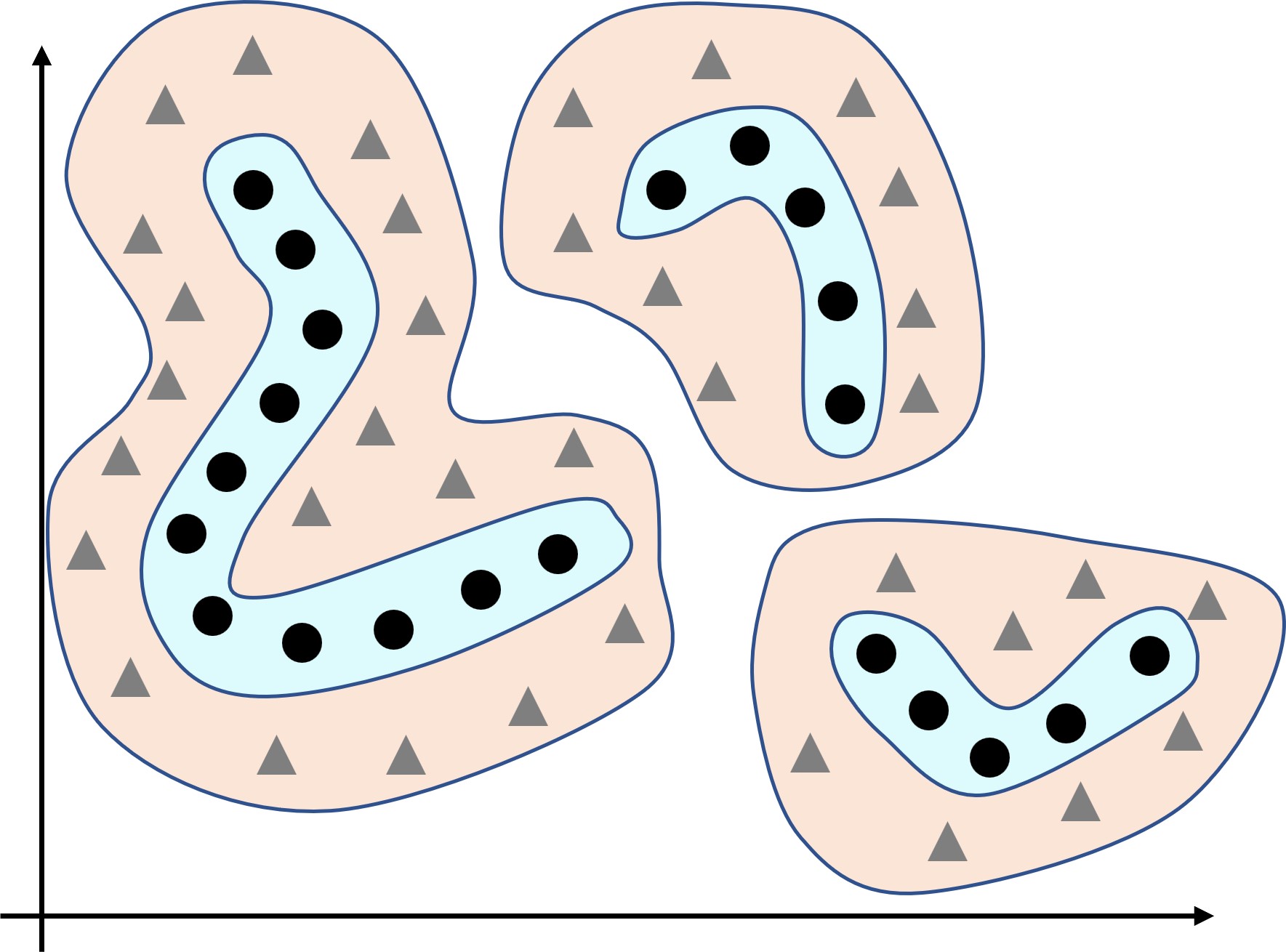}
			\caption{Contrastive learning}
			\label{fig:conLearning}
		\end{subfigure}
		\caption{Comparison of two loss function}
		\label{fig:dcvLearning}
	\end{figure}
	{Fig. \ref{fig:dcvLearning} demonstrates the distinctions in the information learned by the two loss functions. The black dots in the figures represent the positive samples, i.e., the optimal data. The gray triangles denote the negative samples. The light blue regions indicate the areas learned to have $c=0$. The pink regions represent the areas that the required negative samples aim to capture, where $c$ takes a specified value.}
	
	{Fig. \ref{fig:regLearning} indicates that training using \eqref{eq:dsocLoss} can only capture the information in the vicinity of the samples, as the loss function only imposes constraints on the sample points and their derivatives. Regions not covered by the samples lack relevant information. }
	
	{Fig. \ref{fig:conLearning} shows that training using \eqref{eq:dsocLoss2} introduces negative samples, i.e., control inputs that do not match the optimal trajectory, in addition to the optimal point data. This allows the model to learn more global information, beyond just the sample points and their neighborhoods. }
	\section{Case study}
	\label{sec:CS}
	It is worth noting that dynamic optimization problems are often intractable to directly solve when  $L =\infty$ or the terminal time $L$ is not fixed. For infinite-horizon problems, only LQR problems admit analytical optimal solutions, while general nonlinear processes typically lack closed-form optima. Common approaches are based on approximating the cost-to-go via Bellman's optimality principle, then solving a fixed finite-horizon problem. Model predictive control and approximate dynamic programming exemplify such methods. For problems with free terminal time, the switching or final time is usually introduced as a decision variable by input parametrization, transforming the problem into a nonlinear program. Current research on dynamic self-optimizing control has yet to address these two problem classes.
	
	In this article, the following cases are studied to demonstrate DCV design:
	
	\begin{enumerate}[\qquad(a)]
		
		\item A toy example showing DCVs are better suited for multi-valued and discontinuous functions compared to explicit function forms.
		
		\item A batch reaction process illustrating DCV design for problems without fixed time domains.
		
		\item Two stirred tank reactor examples for DCV design in infinite-horizon problems.
		
	\end{enumerate}
	
	The case studies aim to address key dynamic optimization problem classes and highlight the advantages of the DCV approach. 
	{All the results presented in this paper were computed on a computer with a 2.90GHz processor and 8.00GB of RAM. All the results are in https://github.com/Daakuang/DCV-SOC.}
	
	\subsection{Case 1: Toy case}
	
	This case study is constructed to compare the performance of DCV's implicit inference ($\arg_y h(x,y)=0$) and explicit inference ($y=f(x)$) using neural networks. The goal is to approximate discontinuous and multivalued functions, which are challenging for standard explicit inference.
	
	2D sample points $(x, y)$ are generated from two target functions:
	\begin{enumerate}
		\item Multivalued function: $x^2 - y^2 = 1$
		\item Discontinuous function:
		$$
			y = \begin{cases}
				3x + 10, & x \geq 0 \\
				3x - 10, & x < 0
			\end{cases}
		$$
	\end{enumerate}
	
	For function 1, 500 points are sampled in the domain $x \in [-10,-1] \cup [1,10]$, $y=\pm \sqrt{x^2-1}$. For function 2,  500 points are sampled in the domain $x \in [-10,10]$.
	
	{It is trained two fully-connected neural networks with 3 hidden layers of width 10 using the Adam optimizer using Eq.~\eqref{eq:dsocLoss}.} 
	
	For implicit inference, the network predicts $h(x,y)$ and is trained to minimize $\left|\frac{\partial h}{\partial x} h(x,y) \right|$. For explicit inference, the network predicts $f(x)$ and is trained to minimize MSE between $y$ and $\hat{y} = f(x)$. Both networks are trained for 5000 epochs.
	
	The inference results are visualized in Fig.~\ref{fig:toyDCV}. Fig.~\ref{fig:c1} and Fig.~\ref{fig:c2} contains training samples (blue X), predicted values (orange +), and the background color indicates the value of $c$.  Fig.~\ref{fig:c1c} and Fig.~\ref{fig:c2c} contains training samples (blue X), predicted values (orange +).
	\begin{figure}[h]
		\centering
		\begin{subfigure}[b]{0.45\textwidth}
			\centering
			\includegraphics[width=\textwidth]{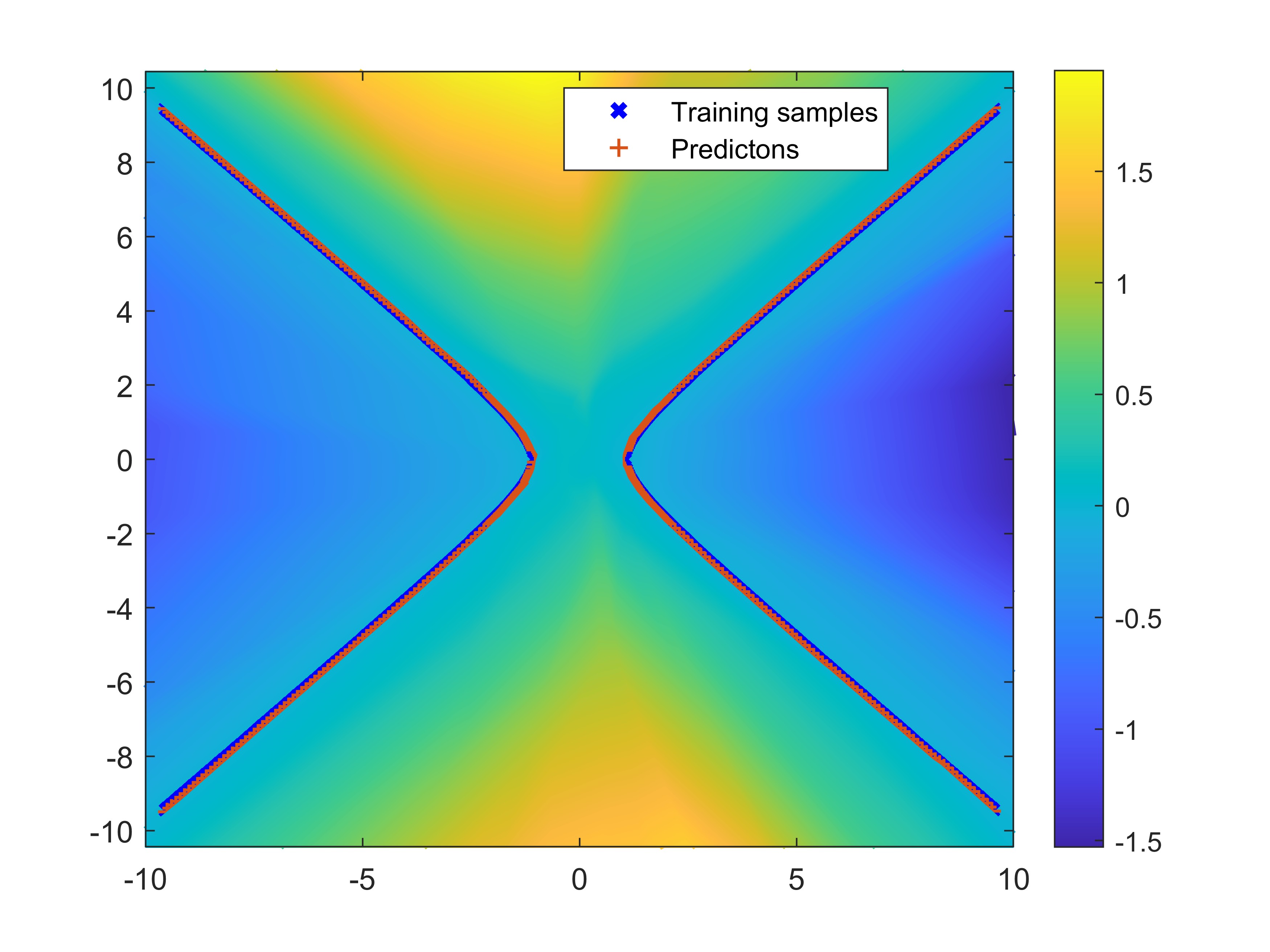}
			\caption{DCV: $x^2-y^2=1$}
			\label{fig:c1}
		\end{subfigure}
		\begin{subfigure}[b]{0.45\textwidth}
			\centering
			\includegraphics[width=\textwidth]{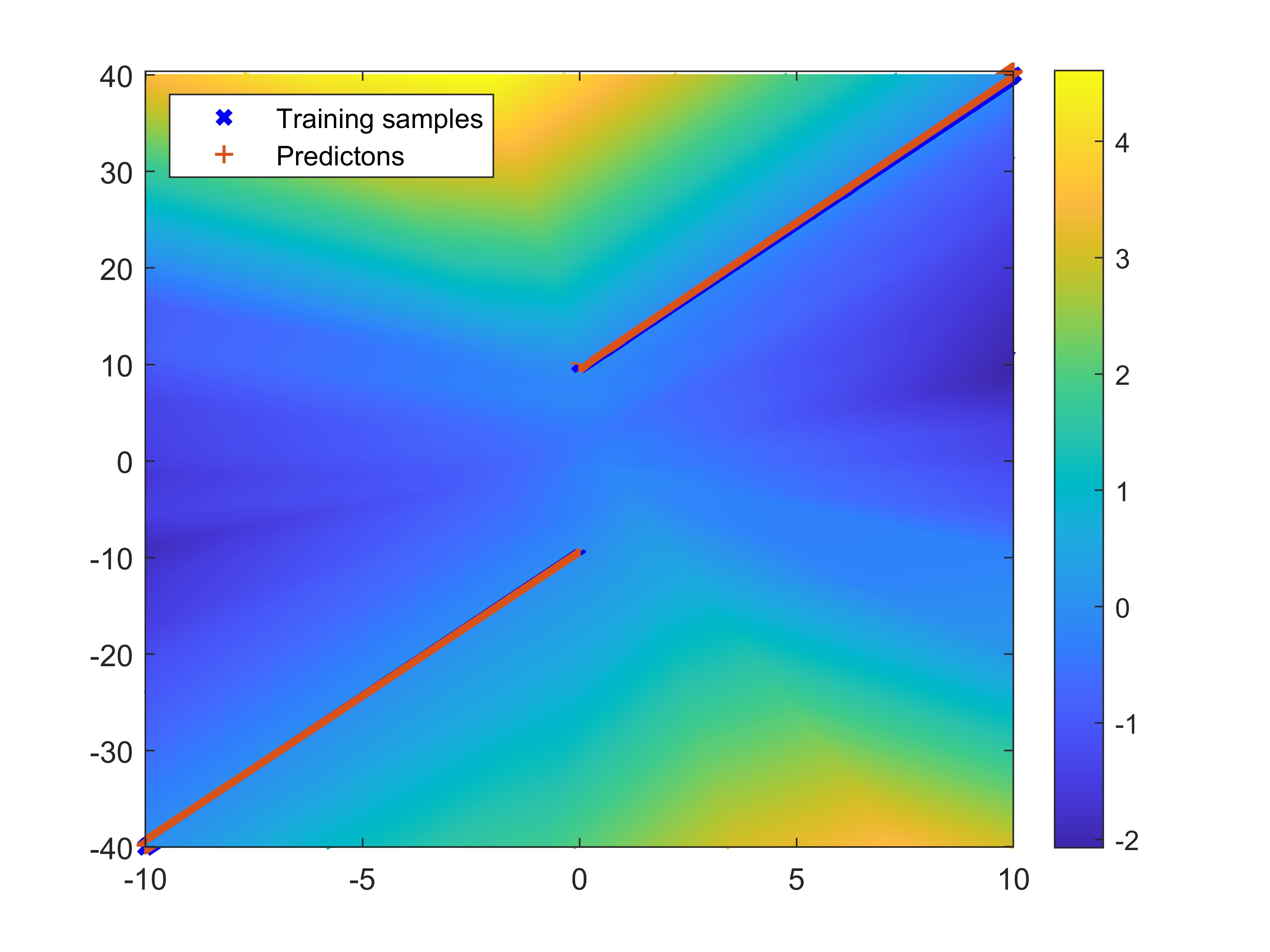}
			\caption{DCV: $y=3x+10,x\geq0;$ $ y=3x-10,x\leq0$}
			\label{fig:c2}
		\end{subfigure}
		\begin{subfigure}[b]{0.45\textwidth}
			\centering
			\includegraphics[width=\textwidth]{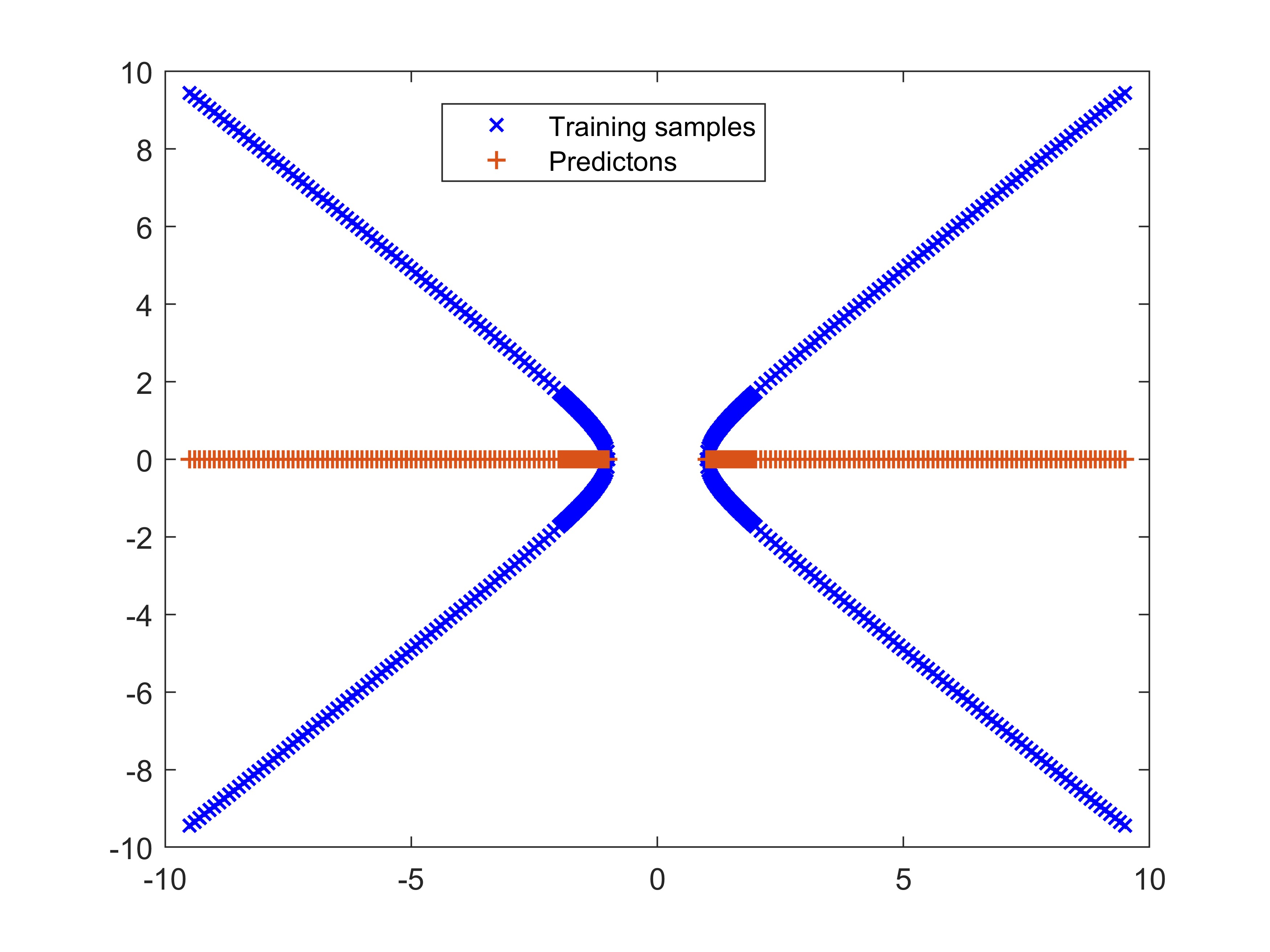}
			\caption{Explicit: $x^2-y^2=1$}
			\label{fig:c1c}
		\end{subfigure}
		\begin{subfigure}[b]{0.45\textwidth}
			\centering
			\includegraphics[width=\textwidth]{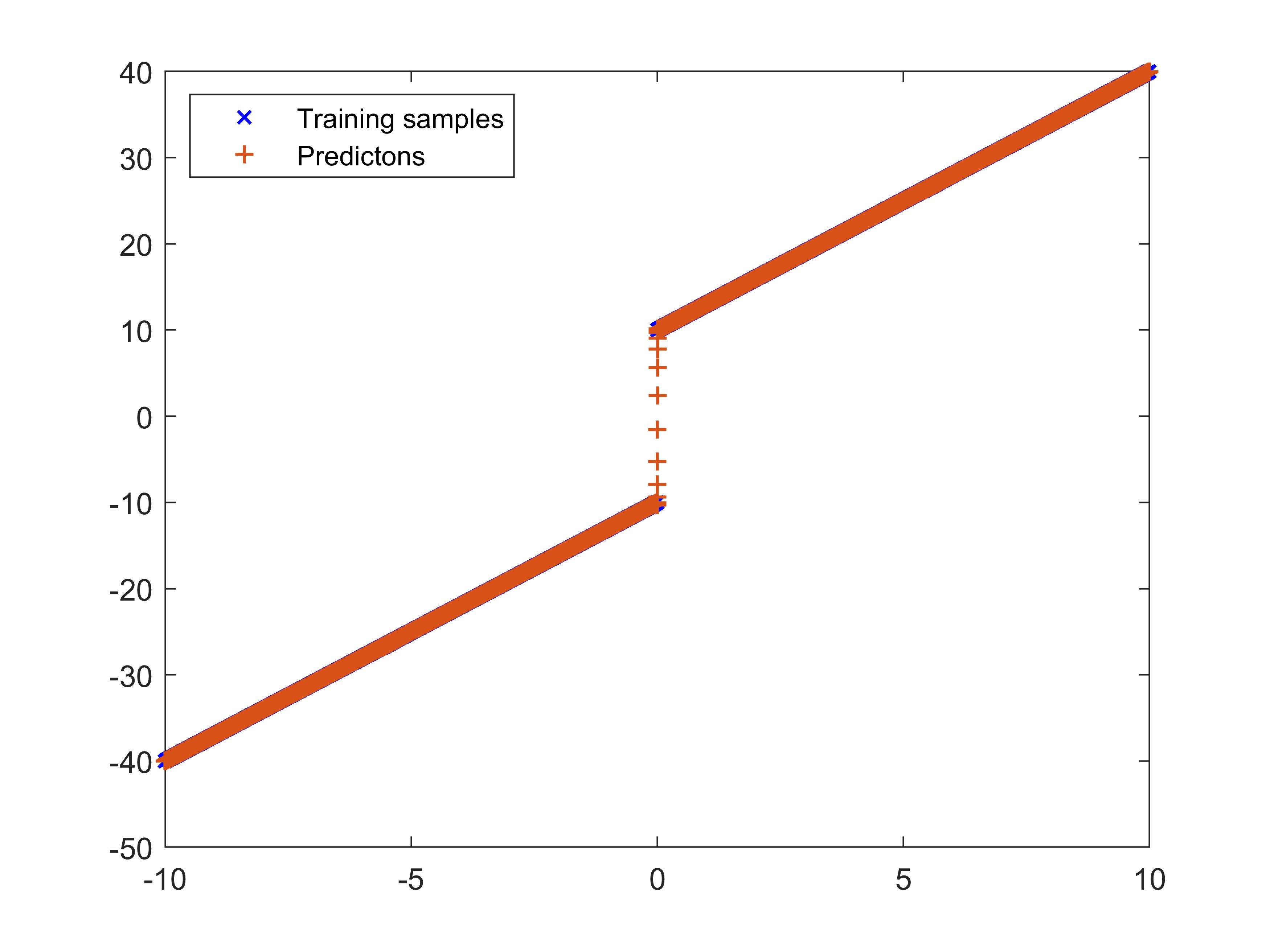}
			\caption{Explicit: $y=3x+10,x\geq0;$ $ y=3x-10,x\leq0$}
			\label{fig:c2c}
		\end{subfigure}
		\caption{Comparison between implicit and explicit form}
		\label{fig:toyDCV}
	\end{figure}
	
	For the multivalued function, implicit inference can successfully approximate the function (Fig.~\ref{fig:c1}), while explicit inference fails(Fig.~\ref{fig:c1c}). For the discontinuous function, implicit inference can capture the discontinuity(Fig.~\ref{fig:c2}), while explicit inference smoothes it out(Fig.~\ref{fig:c2c}).
	
	This demonstrates that by using an implicit constraint formulation, DCV can approximate a wider range of functions compared to standard explicit inference. The implicit formulation does not enforce continuity or single-valued outputs.
	
	It was shown on two numerical examples that DCV with implicit neural network inference was outperformed by explicit inference on challenging multivalued and discontinuous functions.
	This highlights the representational advantages of DCV for control-relevant modeling tasks.

	\subsection{Case 2: DCV for Batch Reactor Optimization}
	\subsubsection{Process description}
	A batch reactor with the exothermic reaction A + B → C is considered \cite{ubrich1999}. The reactor is assumed to be semi-batch, isothermal with cooling. The objective is to minimize the batch time $t_f$ needed to produce a specified amount of product C by optimizing the feed rate of reactant B.
	
	The system dynamics $\mathcal{M}_2$ are governed by material balances for A, B and volume V:
	\begin{align*}
		\dot{c}_A &= -k c_A c_B - \frac{u}{V} c_A \\
		\dot{c}_B &= -k c_A c_B + \frac{u}{V}(c_{B\text{in}} - c_B) \\
		\dot{V} &= u
	\end{align*}
	where, $\dot{c}_A$, $\dot{c}_B$ denote rate of change of concentration of A, B. $c_A$, $c_B$ denote concentration of A, B. $k$ is the reaction rate constant, $u$ is the feed rate of B, $V$ is reactor volume, and $c_{B\text{in}}$ is concentration of B in feed.
	The concentration of product C is:
	\begin{align*}
		c_C = \frac{c_{A 0} V_{0} + c_{C 0} V_{0} - c_A V}{V}
	\end{align*}
	where $c_{A 0}$,$c_{B 0}$,$c_{C 0}$ denote the initial concentration of A, B and C and $V_{0}$ denotes the initial volume.
	The maximum attainable temperature under cooling failure is:
	$$
	T_{\mathrm{cf}}(t)=T+\min \left(c_A(t), c_B(t)\right) \frac{(-\Delta H)}{\rho c_{\mathrm{p}}}.
	$$
	where $T$ is the reaction temperature, $\Delta H$ is the reaction enthalpies, $\rho$
	and $c_{\mathrm{p}}$ denote the density and  heat capacity of the reaction system 
	In this case, the considered measurements only including the three system states
	are represented as:
	$$
		y(k) = [c_A(k) \ c_B(k) \ V(k)]^{\top} 
	$$
	The dynamic optimization problem is as follows:
			$$\begin{aligned}
				 & \min _{t_{\mathrm{f}}, u(t)} J=t_{\mathrm{f}}                        \\
				 & \text { s.t. } \text{system dynamics } \mathcal{M}_2 \\
				 & \qquad u_{\min } \leq u(t) \leq u_{\max }                                   \\
				 & \qquad T_{\mathrm{cf}}(t) \leq T_{\max }                                    \\
				 & \qquad V\left(t_{\mathrm{f}}\right) \leq V_{\max }                          \\
				 & \qquad n_C\left(t_{\mathrm{f}}\right) \geq n_{C \mathrm{des}}
			\end{aligned}$$
	where $u_{\min }$ and $u_{\max }$ represent the minimum and maximum values for the feed flow. $T_{\max }$ and $V_{\max }$ stands for maximum attainable temperature and volume. $n_{C \mathrm{des}}$ denotes the required product output of a batch.
	
	The model parameters, together with their nominal values are given in Table \ref{tab:case2}
	\begin{table}[h]
		\caption{Model parameters, operating bounds and initial conditions for Case 2}
		\label{tab:case2}
		\centering
		\begin{tabular}{lll}
			\hline
			Variable      & Nominal Value  &  Unit \\
			\hline
			$k$               & 0.0482 & $\mathrm{l}/\mathrm{mol}\mathrm{h}$  \\
			$T$               & 70     & ${}^{\circ}\mathrm{C}$     \\
			$\Delta H$         & -60000 & $\mathrm{~J}/\mathrm{mol}$ \\
			$\rho$            & 900    & $\mathrm{~g}/\mathrm{l}$   \\
			$c_p$             & 4.2    & $\mathrm{~J}/\mathrm{gK}$  \\
			$c_{B\text{in}}$  & 2      & $\mathrm{~mol}$             \\
			$u_{\min}$        & 0      & $\mathrm{l}/\mathrm{h}$              \\
			$u_{\max}$        & 0.1    & $\mathrm{l}/\mathrm{h}$              \\
			$T_{\max}$        & 80     & ${}^{\circ}\mathrm{C}$     \\
			$V_{\max}$        & 1      & $\mathrm{l}$                          \\
			$n_{C\text{des}}$ & 0.6    & $\mathrm{~mol}$            \\
			$c_{A0}$ & 2      & $\mathrm{~mol}/\mathrm{l}$          \\
			$c_{B0}$ & 0.63   & $\mathrm{~mol}/\mathrm{l}$ \\
			$V_{0}$  & 0.7    & $\mathrm{l}$                          \\
			\hline                        
		\end{tabular}
	\end{table}
	\subsubsection{Dataset Generation and DCV training}
	
	In dynamic process optimization, converting the complex, continuous problem into a finite-dimensional nonlinear program via control vector parameterization is a well-established technique. As \cite{srinivasan2003a} demonstrates, the optimal input for this batch reaction comprises two distinct phases: $u_{path}$, implemented from $0$ to $t_s$, and $u_{min}$, employed from $t_s$ to the terminal time $t_f$. The switching time $t_s$ and $t_f$ become decision variables alongside the $2N$ input values defining $u_{path}$ through piecewise linear approximation. Numerical investigations reveal that using $N=20$ segments yields sufficient accuracy for the optimized solution.  The trajectories are resampled at a 1 h interval to generate corresponding optimal dataset.
	The reaction rate constant $k$ is treated as a disturbance which follows a uniform distribution, fluctuating $\pm 20\%$ around its nominal value.
	In this case, implicit forecasting form and explicit forecasting form DCV are designed and $M=1$, $H=1$.
	{A fully-connected neural networks with 7 inputs, 2 hidden layers of width 10 and 1 output is trained using the Adam optimizer. }
	\subsubsection{Closed-loop performance analysis}
	
	\begin{figure}[h]
		\centering
		\includegraphics[width=0.8\textwidth]{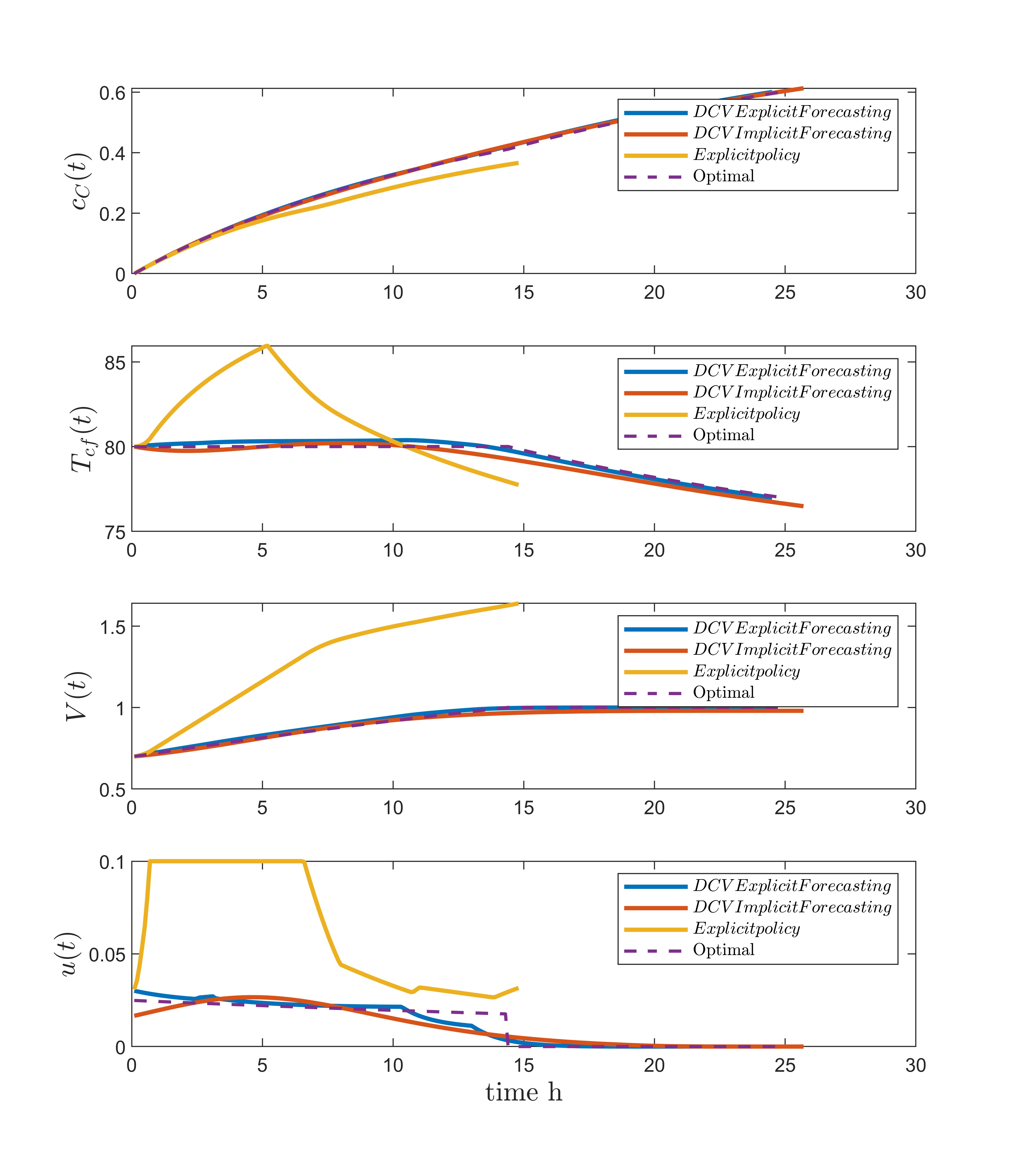}
		\caption{Closed-loop simulation under $-20\%$ disturbance}
		\label{fig:case2dy}
	\end{figure}
	Fig~\ref{fig:case2dy} shows the closed-loop simulations performed under $-20\%$ disturbances.
	It can be seen that the DCV is relatively close to the optimal input.
	
	Closed-loop simulations of DCV were conducted under 10 different disturbance scenarios.
	The results are summarized in Table \ref{tab:cl_results}, showing the batch times and constraint violations. 
	{The table presents the average batch production time, defined as the duration required to produce a sufficient quantity of the desired product, i.e., the time at which $n_C(t) = n_{C \mathrm{des}}$. The path constraint is expressed as $T_{\mathrm{cf}}(t) \leq T_{\max}$, while the terminal constraint is $V(t_{\mathrm{f}}) \leq V_{\max}$. The average violation ratio represents, for each disturbance scenario, the average of the maximum constraint violation relative to the constraint upper bound across all scenarios, computed as: $$\frac{1}{N}\sum\limits_{i=1}^{N}\frac{\max(\max\limits_{t}(g_i(t))-g_{\max},0)}{g_{\max}}$$,where $g_i(t)$ denotes the constraint value at the t-th instant in the i-th scenario, $g_{\max}$ stands for the upper limit of the constraint.
	The maximum violation value signifies the maximum constraint violation across all disturbance scenarios, computed as: $$\max\limits_{i=1,...,N}\max(\max\limits_{t}(g_i(t))-g_{\max},0).$$}
	
	\begin{table}[ht!]
		\caption{Closed-loop results under disturbances}
		\centering
		\label{tab:cl_results}
		\begin{tabular}{lccccc}
			\hline
			\multicolumn{1}{c}{} & \multirow{2}{*}{Time} & \multicolumn{2}{c}{Path Constrains   Violations} & \multicolumn{2}{c}{Terminal Constrains   Violations} \\
			\multicolumn{1}{c}{}     &       & average ratio & maximum value & average ratio & maximum value \\ \hline
			DCV Explicit Forecasting & 20.46 & 0.12\% & 0.3848 & 0.01\% & 0.00027 \\
			DCV Implicit Forecasting & 22    & 0.03\% & 0.2063 & 0.00\% & 0       \\
			Explicit policy          & 19.2  & 3.54\% & 5.9521 & 3.54\% & 0.6398  \\
			Optimal                  & 20.10 & -     & -      & -     & -         \\ \hline
		\end{tabular}
	\end{table}
	{Regarding constraint satisfaction, the explicit policy exhibits substantial constraint violations, pertaining to both path and terminal constraints. Consequently, although it achieves the shortest batch production time, this is attained at the expense of violating constraints. In contrast, the implicit policy based on DCV adheres to constraints with negligible violations. The explicit predictive DCV approach exhibits a slight inferiority in constraint violation compared to the implicit predictive DCV, but as evident from Fig. \ref{fig:case2dy}, their discrepancy in constraint handling is virtually negligible. However, in terms of production time, the explicit predictive DCV demonstrates a distinct advantage over the implicit predictive DCV.}
	This demonstrates the ability of implicit DCV to accurately represent discontinuous optimized profiles, enabling improved performance over an explicit formulation. The case study highlights the benefits of DCV for self-optimizing control of complex dynamic processes.

	\subsection{DCV for infinite time domain optimal control}
	This section demonstrates two cases to illustrate the applicability of self-optimizing DCVs to infinite-horizon optimal control problems, including a multimodal optimal control problem. The two cases and how the optimal data is generated are introduced below.
	\subsubsection{Case 3: Benchmark continuous stirred-tank reactor}
	The proposed methodology is first applied to a benchmark continuous stirred-tank reactor (CSTR) problem from \cite{mayne2011}. This problem comprises two states: the scaled concentration and reactor temperature, denoted by $x_1$ and $x_2$, respectively. The process is controlled by the coolant flow rate $u$. The model is:
	\begin{align*}
		\dot{x}_1 &= \frac{1}{\tau}(1-x_1) - kx_1\exp\left(-\frac{\beta}{x_2}\right) \\
		\dot{x}_2 &= \frac{1}{\tau}(x_f-x_2) + kx_1\exp\left(-\frac{\beta}{x_2}\right) - \alpha u(x_2-x_c)
	\end{align*}
	with model parameters $\tau=20$, $k=300$, $\beta=5$, $x_f=0.3947$, $x_c=0.3816$, and $\alpha=0.117$. The state and input constraints are $\mathcal{X}=[0.0632, 0.4632] \times [0.4519, 0.8519]$ and $\mathcal{U}=[0,2]$, respectively. The setpoint is $x^{\text{sp}}=[0.2632, 0.6519]^\top$ and $u_e = 0.7583$ . The stage cost is:
	$$
	\ell(x,u) = |x-x^{\text{sp}}|^2 + 10^{-4}|u-u_e|^2
	$$
	The problem is solved with a sampling time of 3 s and a prediction horizon of $L=140$.
	In this case, the learning samples are generated using Algorithm \ref{alg:DCV_data}, with parameters set as: $N = 1600$, $H = 0$, $M = 1$. And a grid-based sampling method as in \cite{hertneck2018} is used.
	The optimal control problem is solved with Casadi to produce optimal trajectory data.
	
	\subsubsection{Case 4: a multimodal optimal control problem}
	The second case is a multimodal optimal control problem used by Luus \cite{luus2019}
	$$
	\begin{aligned}
		\dot{x}_1= & -(2+u)\left(x_1+0.25\right) + \left(x_2+0.5\right) \exp \left(\frac{25 x_1}{x_1+2}\right), \\
		\dot{x}_2= & 0.5-x_2-\left(x_2+0.5\right) \exp \left(\frac{25 x_1}{x_1+2}\right),
	\end{aligned}
	$$
	where $x_1$ represents the deviation from dimensionless steady-state temperature and $x_2$ stands for the deviation from the dimensionless steady-state concentration. The control $u(t)$ represents the manipulation of the flow-rate of the cooling fluid, which is inserted in the reactor through a coil.
	 $\mathcal{X}_0=[-0.09, 0.09] \times [-0.09, 0.09]$ and $\mathcal{U}=[-1,5]$. The optimal control problem is to determine the $u(t)$ that minimises the performance index:
	$$
	\mathcal{J}=\int_0^{t_f}\left(x_1^2+x_2^2+0.1 u^2\right) \d t
	$$
	The problem is solved with a sampling time of 0.05 s and a prediction horizon of $L=150$.
	
	Fig~\ref{fig:case3mul} shows two modal solutions with initial state $x_0 = [0.09~0.09]^\top$.
	\begin{figure}[ht!]
		\centering
		\includegraphics[width=0.6\textwidth]{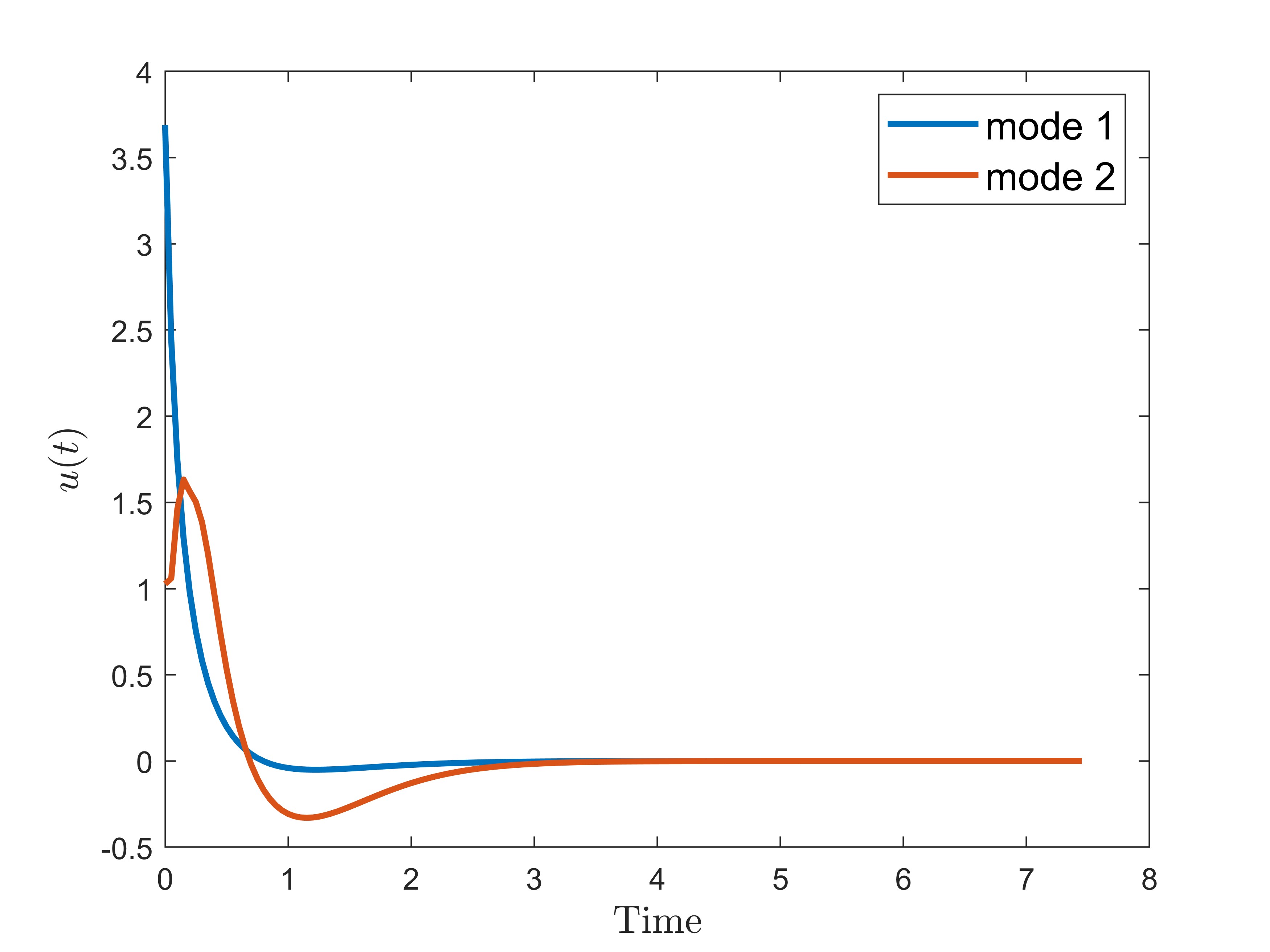}
		\caption{Multimodal solution case 3}
		\label{fig:case3mul}
	\end{figure}
	It can be seen that the trajectories are completely different.
	
	In this case, the learning samples are generated using Algorithm \ref{alg:DCV_data}, with parameters set as: $N = 1600$, $H = 0$, $M = 1$. And a grid-based sampling method as in \cite{hertneck2018} is used.
	Multiple modal solutions are obtained by solving the optimal control problem with multiple initial values. A total of 1600 globally optimal solutions and 1115 locally optimal solutions were obtained in 1600 disturbance scenarios. For training the control policy, 500 scenarios with local optimal solutions were randomly selected, and the data for those disturbance scenarios were removed from the dataset of global optimal solutions. 1100 global optimal solutions and 500 local optimal solutions form the final dataset. The purpose is to simulate the case where a certain proportion of non-globally optimal solutions may exist when solving optimal control problems.
	{In case 3, a fully-connected neural networks with 3 inputs, 2 hidden layers of width 10 and 1 output is trained using Levenberg-Marquardt backpropagation method \cite{more1978}. 
	In case 4, a fully-connected neural networks with 3 inputs, 2 hidden layers of width 10 and 1 output is trained using Levenberg-Marquardt backpropagation method.}
	\subsubsection{Results analysis}

	\begin{table}[ht!]
		\centering
		\caption{Closed loop performance comparison}
		\label{tab:case3}
		\begin{tabular}{@{}llll@{}}
			\hline
			& \multicolumn{1}{c}{\multirow{2}{*}{Case 3}} & \multicolumn{2}{c}{Case 4} \\ 
			& \multicolumn{1}{c}{}                       & one modul    & 2 moduls    \\ 
			\hline 
			Optimal solution             & 0.07858                                    & 0.01632      & 0.01632     \\
			DCV             & 0.07863                                    & 0.03689      & 0.03689     \\
			Explicit policy & 0.07863                                    & 0.03689      & 0.11115     \\ \hline
		\end{tabular}
	\end{table}
	{In Table~\ref{tab:case3}, the DCV is compared to an explicit control policy trained with MSE loss over 160 disturbance scenarios. The table presents the average of $\Delta \mathcal{J} = \mathcal{J} - \mathcal{J}^\star$ across all scenarios, calculated as: $\frac{1}{N}\sum\limits_{i=1}^{N}(\mathcal{J}_i - \mathcal{J}_i^\star)$, where ${J}$ represents the actual cost functional achieved and ${J}^\star$ denotes the optimal cost.}
	Case 4 one modal stands for the training dataset containing only global optimal solutions, while Case 4 two modals means the training samples include both local and global optimal solutions. Since Case 3 has only one modal solution, the optimal solutions are used.
	It can be seen that when multimodal solutions exist in the training samples, the DCV significantly outperforms the explicit policy. In problems with only a single modal solution, the two have very close performance.
	
	\begin{figure}[ht]
		\centering
		\includegraphics[width=0.6\textwidth]{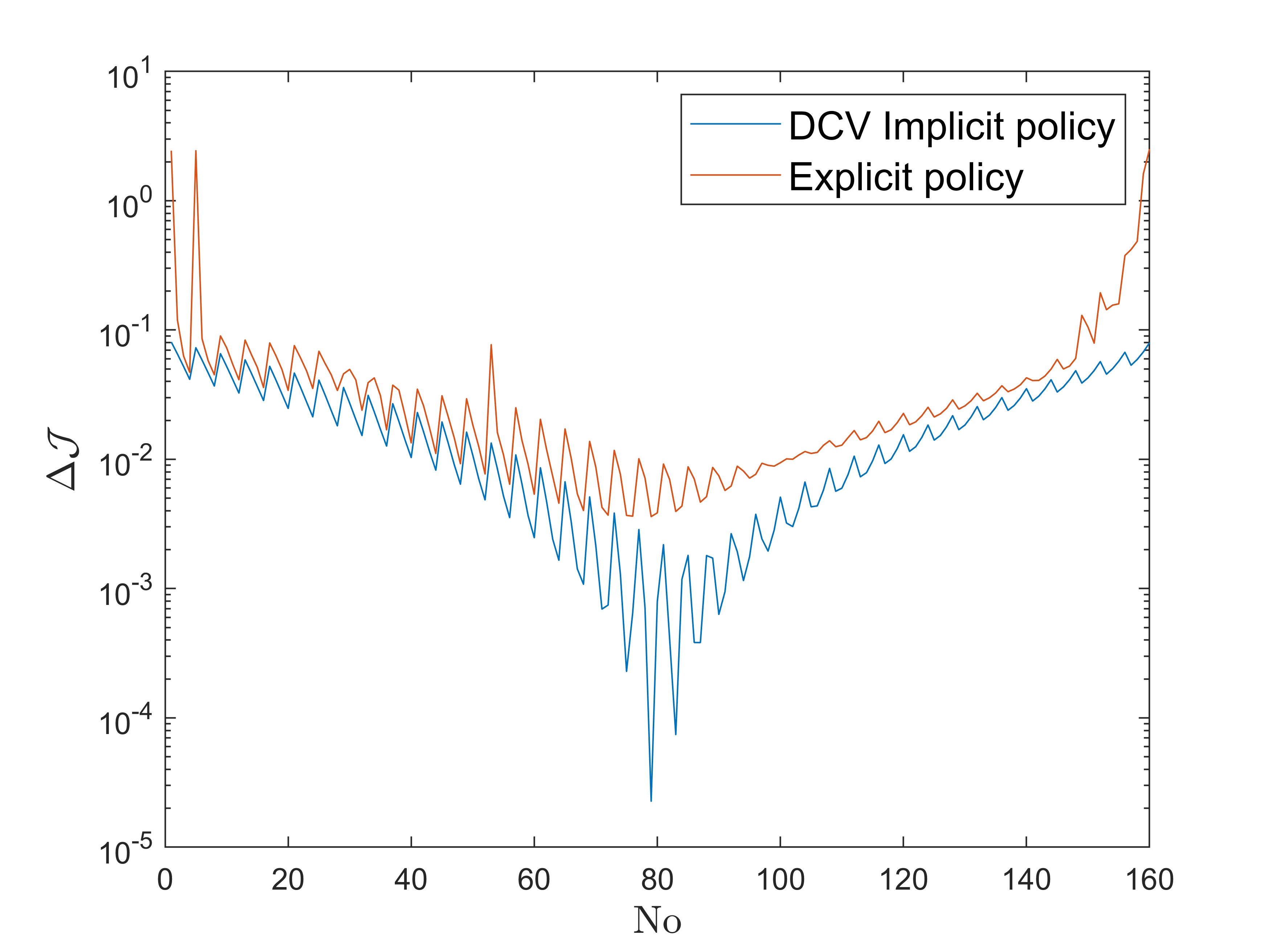}
		\caption{Simulation results in Case 4  multimodals}
		\label{fig:case3}
	\end{figure}
	
	Furthermore, Fig.~\ref{fig:case3} shows the closed-loop loss across different scenarios in case 4 two modals. {The ordinate means the value of $\Delta \mathcal{J} = \mathcal{J} - \mathcal{J}^\star$. The abscissa represents the serial number of different scenes.}
	It can be seen that the explicit policy has very high closed-loop loss in some scenarios, while the DCV performs consistently better across all scenarios. 
	
	%
	{\begin{table}[]
		\centering
		\caption{Inference time in Case 4 }
		\label{tab:timei}
		\begin{tabular}{@{}llll@{}}
			\hline
			& Optimization & DCV   & Explicit policy \\ \hline
			Inference time/per(ms) & 1021.8       & 45.18 & 3.913           \\ \hline
		\end{tabular}
	\end{table}
	DCV involves online root-finding of equations, necessitating a certain computational cost. Here in Table \ref{tab:timei}, the single-step inference times of DCV, explicit policy, and direct solution of the optimal control problem are compared. The results presented in this table were computed on a computer with a 2.90GHz processor and 8.00GB of RAM. The data in the table represents the average inference time required under different initial states and initial values. It can be observed that the online solution time of DCV is significantly shorter than the time required for directly solving the optimization problem, although the explicit policy time is slightly shorter.}
	\section{Discussion}
	\subsection{Comparison which model predictive control policy approximation}
	{Model predictive control policy approximation \cite{zhou2024a} employs deep neural networks to approximate the control laws predicted by models, garnering extensive attention due to its capability to substantially reduce the online computational demands of model predictive control, particularly for nonlinear systems. Within the framework considered in this paper, model predictive control policy approximation adopts explicit control policies to approximate the optimal control strategy. Broadly speaking, the method proposed herein and model predictive control policy approximation both fall under the purview of imitation learning \cite{florence2022a}. Consequently, the DCV and the implicit policies based on DCV proposed in this paper potentially hold application in imitation learning. }

	\subsection{Comparison of application scope of state-of-art methods}
	{The distinctions between the proposed method and state-of-the-art approaches in terms of scope and computational load are summarized in Table \ref{tab:com}.}
	\begin{table}[ht]
		\caption{Comparison of application scope and computational load of state-of-art methods}
		\label{tab:com}
		\centering
		\begin{tabular}{clcccccc}
			\hline
			\multicolumn{1}{l}{} &
			&
			\textbf{RTO} &
			\textbf{NMPC} &
			\textbf{SSOC} &
			\textbf {DSOC} &
			\textbf{MPCA} &
			\textbf{DCV-SOC} \\ \hline
			\multicolumn{1}{c|}{\multirow{2}{*}{\textbf{\begin{tabular}[c]{@{}c@{}}Computation\\ Load\end{tabular}}}} &
			Online &
			Middle &
			High &
			Low &
			Low &
			Low &
			Low \\
			\multicolumn{1}{c|}{} & Offline          & -          & -          & Middle     & High       & High       & High       \\ \hline
			\multicolumn{1}{c|}{\multirow{3}{*}{\textbf{\begin{tabular}[c]{@{}c@{}}Dynamic\\ Optimization\end{tabular}}}} &
			Fixed-horizon &
			\XSolid &
			\Checkmark &
			\XSolid &
			\Checkmark &
			\Checkmark &
			\Checkmark \\
			\multicolumn{1}{c|}{} & Infinite-horizon & \XSolid    & \Checkmark & \XSolid    & \XSolid    & \Checkmark & \Checkmark \\
			\multicolumn{1}{c|}{} & Free-horizon     & \XSolid    & \Checkmark & \XSolid    & \XSolid    & \Checkmark & \Checkmark \\ \hline
			\multicolumn{1}{c|}{\multirow{4}{*}{\textbf{\begin{tabular}[c]{@{}c@{}}Control\\ Law \\ Mapping\end{tabular}}}} &
			Continuous &
			\Checkmark &
			\Checkmark &
			\Checkmark &
			\Checkmark &
			\Checkmark &
			\Checkmark \\
			\multicolumn{1}{c|}{} & Discontinuous    & \Checkmark & \Checkmark & \XSolid    & \XSolid    & \XSolid    & \Checkmark \\
			\multicolumn{1}{c|}{} & One-to-one       & \Checkmark & \Checkmark & \Checkmark & \Checkmark & \Checkmark & \Checkmark \\
			\multicolumn{1}{c|}{} & Multivalued      & \Checkmark & \Checkmark & \XSolid    & \XSolid    & \XSolid    & \Checkmark \\ \hline
		\end{tabular}
	\end{table}
	 {RTO denotes real-time optimization, and NMPC represents nonlinear model predictive control, both of which require online solution of optimization problems, resulting in relatively higher online computational demands compared to SOC. However, since their control policies are obtained through online solutions, they can represent discontinuous and multi-valued control laws. SSOC and DSOC stand for steady-state self-optimizing control and traditional dynamic self-optimizing control, respectively, neither of which can address dynamic optimization problems with infinite or free time horizons. MPCA denotes model predictive control policy approximation, while DCV-SOC represents dynamic self-optimizing control based on DCV, the method proposed in this paper. Although MPCA can also solve the dynamic optimization problems considered herein, its adoption of explicit control policies precludes the representation of discontinuous and multi-valued functions\cite{li2022}.}
	\subsection{Relationship between DCV and energy-based policy}
	It is worth noting that the idea of using this implicit function form to approximate a mapping is similar to energy-based models(EBM)\cite{lecun2006}. 
	Energy-based models define an energy function that assigns scalar energy values to each point in the input space. The energy surface captures dependencies between inputs and outputs, with low energies near optimal or high-probability regions of the space. The energy function is parameterized and trained to match expected properties of the space, like minimizing energies at target outputs.
	Energy-based control policy have been successfully applied in robotics and theoretically proven to be better suited for approximating multi-valued functions \cite{ghaouiImplicitDeepLearning2020,pan2020}. DCV also can be treated as a novel form of EBM.
	
	The differences and similarities between the DCV form in this paper and energy functions can be summarized in the following table \ref{tab:DCV&EBM}:
	\begin{table}[th!]
		\caption{The difference between DCV and Energy based policy}
		\label{tab:DCV&EBM}
		\centering
		\begin{tabular}{lll}
			\hline
			\textbf{Feature} & \textbf{Dynamic controlled variables} & \textbf{Energy based policy} \\ \hline
			Function form & $h_{\theta}(u_k,\boldsymbol{\xi}_{[k-M,k-1]})$ & $ E_{\theta}(u_k,\boldsymbol{\xi}_{[k-M,k-1]})$ \\
			Inference & $\arg_{u_k} h_{\theta}(u_k,\boldsymbol{\xi}_{[k-M,k-1]})=0 $ & $\arg \min_{u_k} E_{\theta}(u_k,\boldsymbol{\xi}_{[k-M,k-1]})=0 $ \\
			Output & Controlled variable $c_k \in \mathbb{R}^{n_u}$ & Energy value $E \in \mathbb{R}$ \\ \hline
		\end{tabular}
	\end{table}
	Potentially, the advantage of DCVs over EBM is that multiple controlled variables can be decoupled and calculated separately. This is attractive for distributed systems, multi-agent systems, etc.
%
	\subsection{Co-design for DCV and controller}
	Because the form of DCVs is very flexible, almost all current research on explicit control strategies can be integrated into the framework of DCVs to realize the collaborative design of controlled variables and controllers, which increases the flexibility of existing controller design methods. The collaborative design of controllers and DCVs to enhance current control design approaches is also a promising research direction for the future.
	\subsection{Inference of DCV}
	The online inference of DCVs involves solving the DCVs equations online. Different solution methods or initial values may lead to different solutions. How to design or select the online inference approach is also worth investigating in the future. For example, using observed information to update the constraint range of $u_k$ might improve robustness and safety under DCV policies.
	\section{Conclusion}
	{In this paper, the dynamic SOC problem is formally defined for the first time. Unlike traditional SOC problems, emphasis is placed on maintaining the controlled variables at constant values at all times. Based on this, the concept of controlled variables is extended for dynamic self-optimizing control problems, and the notion of dynamic controlled variables (DCVs) is introduced for the first time, enabling controlled variables to better capture the dynamic characteristics of the system. }
	
	{An innovative implicit control strategy based on DCVs is proposed and theoretically proven to be more suitable for representing discontinuous or multivalued control laws. Simultaneously, the concept of DCVs is compared with the traditional notion of controllers based on explicit mappings between measurements and inputs, suggesting a shift in focus towards designing control strategies that actively constrain system output behavior, rather than passive reactive responses. }
	
	{The concept of dynamic SOC is also contrasted with sliding mode control, highlighting the distinctions between the two and emphasizing that dynamic SOC is a method oriented towards control structure design. Subsequently, the requirements for self-optimizing CVs are extended based on the definition of dynamic self-optimizing control, and four metrics are proposed to quantitatively evaluate the optimality of DCVs. }
	
	{Finally, the problem of designing self-optimizing DCVs is viewed as a mapping identification problem, and a general method for designing DCVs for dynamic optimization problems is presented, marking the first solution in the field of self-optimizing control for dynamic optimization problems with non-fixed time horizons. Ultimately, four case studies are conducted to validate the approach.}
	
	{Our results indicate that the DCV is a flexible and effective way to approximate time-varying optimal control policies, and that the implicit policy based on DCVs has advantages in accurately representing discontinuous or multi-modal optimization profiles.}
	Almost all control theory is predicated on designing explicit policy. The implicit policy based on DCVs furnishes a novel perspective and practical apparatus for constructing control policies.
	
	However, our work also has some limitations that need to be addressed in future research:
	\begin{enumerate}
		\item The performance of DCVs depends on the training samples, training method, function approximator architecture, and online inference approach, which need further investigation.
		\item This paper only shows DCVs can be applied to constrained problems, but does not theoretically guarantee safety and stability under the implicit policy based on DCVs. Safety and stability under the implicit policy based on DCVs is also an urgent research topic.
	\end{enumerate}
	Future research should concentrate on developing more robust and efficient techniques for designing and implementing DCV. Additionally, testing and comparing DCVs with other control methods on more realistic and challenging problems should be explored. This includes investigating methods for designing distributed DCVs and how DCVs can handle stochastic systems and systems with distributed parameters.

	\section*{Acknowledgments}
	This work is supported by the Research Funds of Institute of Zhejiang University-Quzhou( IZQ2022KYZX12, IZQ2022KYZX08,IZQ2022KJ3001, IZQ2022KJ1003 ) and key projects of Zhejiang High-end Chemical Technology Innovation Center (ACTIC-2022-013)
	\section*{Appendix}
	\subsection*{Variable description}
	All variables in this paper and their descriptions are in Table \ref{tab:var}.
	\begin{table}[!h]
		\caption{Variable description}
		\label{tab:var}
		\begin{tabular}{llll}
			\hline
			\textbf{Variable}    & \textbf{Description}         & \textbf{Variable} & \textbf{Description}                \\ \hline
			$x$                  & state variavle               & $H$               & predition window                    \\
			$x_0$                & initial state                & $M$               & observation window                  \\
			$u$                  & control input                & $\pi$             & controllor                          \\
			$c$                  & controlled variable          & $J$               & Jacobian matrix                     \\
			$y$                  & measurement                  & $\theta$          & parameter                           \\
			$w$                  & measurement noise            & $G$               & Sensitivity of c to u               \\
			$\xi$                & extend measurement           & $\mathcal{L}$     & loss function                       \\
			$d$                  & disturbance                  & $\tau$            & hyperparameters                     \\
			$g$                  & constrain                    & $\mathfrak{D}$    & data set                            \\
			$f$                  & dynamics equations           & $k$               & reaction rate constant              \\
			$m$                  & measurement equation         & $T$               & reaction temperature                \\
			$h$                  & controlled variable function & $\Delta H$        & reaction enthalpies                 \\
			$\mathcal{X}_0$      & Initial state distribution   & $\rho$            & density                             \\
			$\mathcal{D}$        & Disturbance distribution     & $c_p$             & heat capacity                       \\
			$\mathcal{W}$        & Noise distribution           & $c_{B\text{in}}$  & concentration of B in feed          \\
			$N$                  & Number of samples            & $u_{\min}$        & the lowwer limited of control input \\
			$L$                  & operating time horizon       & $u_{\max}$        & the upper limited of control input  \\
			$\mathcal{J}$        & toall cost function          & $n_{C\text{des}}$ & amount of C                         \\
			$\mathcal{J}^\star $ & optimal cost function        & $c_{A0}$          & initial concentration of A          \\
			$l_i$                & path cost function           & $c_{B0}$          & initial concentration of B          \\
			$l_f$                & terminal cost function       & $V_{0}$           & initial volume                      \\
			$S$                  & set                          & $c_A$             & concentration of A                  \\
			$S_{\mathrm{opt}}$   & optimal solution set mapping & $V$               & reactor volume                      \\ \hline
		\end{tabular}
	\end{table}
	
	\subsection*{Proofs of Theorem \ref{the:DCVF} }
	
	\begin{lemma}\label{the:florence2022a}
		Let $F:\mathbb{R}^m \rightrightarrows \mathcal{Y}\left(\mathbb{R}^n\right)$ be a set-valued function with closed graph, defined as
		\begin{align*}
			\text{gph }F = \{(x, y) \in \mathbb{R}^m \times\mathcal{Y}\left(\mathbb{R}^n\right) \mid y \in F(x)\}
		\end{align*}
		and $F(x) \neq \emptyset$ for all $x \in \mathbb{R}^m$. Then there exists an $L$-Lipschitz continuous function $E:\mathbb{R}^m \times \mathbb{R}^n \rightarrow \mathbb{R}$ such that
		$$
			F(x) = \arg \min_{y \in\mathcal{Y}\left(\mathbb{R}^n\right)} E(x,y), \quad \forall x \in \mathbb{R}^m.
		$$
		and the minimum value of $E(x,y)$ is zero at $x$.
	\end{lemma}
	\begin{proof}
		See \cite{florence2022a} Theorem 1.
	\end{proof}
	
	\begin{lemma}
		\label{lem:eqeq}
		Let $c:\mathcal{X} \times \mathcal{Y} \rightarrow \mathbb{R}^n$ be a continuous function, where $\mathcal{X} \subseteq \mathbb{R}^m$ and $\mathcal{Y} \subseteq \mathbb{R}^n$. Define the solution mappings:
		\begin{align*}
			S_{eq}&: \mathcal{X} \rightrightarrows \mathcal{Y}, \quad S_{eq}(x) = \{y \in \mathcal{Y} \mid c(x,y) = 0\} \\
			S_{opt}&: \mathcal{X} \rightrightarrows \mathcal{Y}, \quad S_{opt}(x) = \arg \min_{y \in \mathcal{Y}} \|c(x,y)\|_2^2
		\end{align*}
		where $|\cdot|_2$ is the Euclidean norm.

		If $S_{eq}(x) \neq \emptyset$ for all $x \in \mathcal{X}$, then $S_{eq}(x) = S_{opt}(x)$ for all $x \in \mathcal{X}$.
	\end{lemma}
	\begin{proof} 
		It can be proven by contradiction.
		
		Assume there exists a pair $(x_1,y_1)$, $ y_1 \in S_{opt}(x_1)$, and $y_1 \notin S_{eq}(x_1)$, then $c(x_1,y_1) \neq 0$, $\|c(x_1,y_1)\|_2^2 > 0$. Since for all $x \in \mathcal{X}$, $S_{eq}(x)$ is not empty, there exists $x_2,y_2$ such that $c(x_2,y_2)=0$, then $\|c(x_2,y_2)\|_2^2=0 < \|c(x_1,y_1)\|_2^2$. This contradicts $(x_1,y_1) \in S_{opt}$. 
		Therefore the assumption does not hold and $ S_{opt}(x) \subseteq  S_{eq}(x)$
		the lemma is proved. 
		
		Assume there exists a pair $(x_3,y_3)$, $ y_3 \in S_{eq}(x_3)$, and $y_3 \notin S_{opt}(x_3)$, then $c(x_3,y_3) = 0$, $\|c(x_3,y_3)\|_2^2 = 0 >S_{min}(x_3)$. Since for any $(x,y)$, $\|c(x,y)\|_2^2 \geq 0$, $\inf_{x} S_{min}(x) = 0$. This contradicts $c(x_3,y_3)^{\top}c(x_3,y_3) = 0 >S_{min}(x_3)$.
		Therefore, the assumption does not hold and $ S_{eq}(x) \subseteq  S_{opt}(x)$
		
		Thus, $ S_{eq}(x) =  S_{opt}(x)$
	\end{proof}
	
	\setcounter{theorem}{0}
	\begin{theorem}
		Let $F:\mathbb{R}^m \rightrightarrows \mathcal{Y} \subseteq \mathbb{R}^n $ be a set-valued mapping with closed graph, defined as
		\begin{align*}
			\text{gph }F = \{(\mathbf{x}, \mathbf{y}) \in \mathbb{R}^m \times \mathcal{Y} \mid \mathbf{y} \in F(\mathbf{x})\}
		\end{align*}
		and $F(\mathbf{x}) \neq \emptyset$ for all $\mathbf{x} \in \mathbb{R}^m$. Then there exists a continuous function $g:\mathbb{R}^m \times \mathbb{R}^n \rightarrow \mathbb{R}^n$ such that the solution mapping
		\begin{align*}
			S:\mathbb{R}^m \rightrightarrows \mathcal{Y}, \quad S(\mathbf{x}) = \{ \mathbf{y} \in \mathbb{R}^n \mid g(\mathbf{x},\mathbf{y}) = 0 \} 
		\end{align*}
		satisfies $S(\mathbf{x}) \neq \emptyset$ for all $\mathbf{x} \in \mathbb{R}^m$ and
		\begin{align*}
			F(\mathbf{x}) = S(\mathbf{x}), \quad \forall \mathbf{x} \in \mathbb{R}^m.
		\end{align*}
	\end{theorem}
	\begin{proof}
		Let $E(x,y)=\|g(x,y)\|$ in Lemma \ref{the:florence2022a}, so $F(\mathbf{x}) = \arg \min_{y \in \mathcal{Y}} \|g(x,y)\|$. And based on Lemma \ref{lem:eqeq},  $\arg \min_{y \in \mathcal{Y}} \|g(x,y)\| = \arg_{y \in \mathcal{Y}} g(x,y)=0$, thus 	$F(\mathbf{x}) = S(\mathbf{x})$
	\end{proof}
	
	\begin{proof}
	By Lemma \ref{the:florence2022a}, there exists an $L$-Lipschitz continuous function $E:\mathbb{R}^m \times \mathbb{R}^n \rightarrow \mathbb{R}$ such that
	\begin{align*}
		F(\mathbf{x}) = \arg \min_{\mathbf{y} \in \mathcal{Y}} E(\mathbf{x},\mathbf{y}), \quad \forall \mathbf{x} \in \mathbb{R}^m.
	\end{align*}
	and the minimum value of $E(x,y)$ is zero at $x$.
	
	Define $ \|g(\mathbf{x},\mathbf{y})\| = E(\mathbf{x},\mathbf{y})$, which is continuous since $E$ is Lipschitz continuous.
	
	Consider the solution mapping $S:\mathbb{R}^m \rightrightarrows \mathcal{Y}$ defined as
	\begin{align*}
		S(\mathbf{x}) = \{\mathbf{y} \in \mathcal{Y} \mid g(\mathbf{x},\mathbf{y}) = 0\},
	\end{align*}
	which $S(\mathbf{x}) \neq \{\emptyset\}$ since the minimum value of $E(x,y)$ is zero at $x$.
	
	By Lemma \ref{lem:eqeq}, $S(\mathbf{x}) = F(\mathbf{x})$ for all $\mathbf{x} \in \mathbb{R}^m$. Therefore, $g$ satisfies the requirements.
	\end{proof}
	

	\bibliographystyle{unsrt}
	\bibliography{biblio}

\end{document}